\documentclass[12pt]{article}

\usepackage{amsbsy,amsfonts,amsmath,amssymb,amsthm
,bbm,enumerate,enumitem,euscript,graphicx,epsfig,psfrag,color}

\usepackage{amsthm}

\newtheorem{theo}{Theorem}
\newtheorem{coro}{Corollary}

\newtheorem{prop}{Proposition}
\newtheorem{lemm}{Lemma}
\newtheorem{defn}{Definition}

\newtheorem{rem}{Remark.}

\def\N{\mathbb{N}}
\def\Z{\mathbb{Z}}
\def\R{\mathbb{R}}

\def\P{\mathbb{P}}
\def\PP{\mathcal{P}}
\def\L{\Lambda}
\def\O{\Omega}
\def\o{\omega}
\def\B{\mathcal{B}}
\def\G{\mathcal{G}}
\def\A{\mathcal{A}}
\def\EE{\mathcal{E}}

\def\LL{\mathcal{L}}
\def\Ncc{N_{cc}}
\def\bc{{\text{bc}}}

\newcommand{\1}{\ensuremath{\mbox{\rm 1\kern-0.23em I}}}

\begin{document}
\title{Fully-connected bond percolation on $\Z^d$ }

\author{David Dereudre
 \footnote{Univ. Lille, CNRS, UMR 8524, Laboratoire Paul Painlev\'e, F-59000  lille, France.}\\
{\normalsize{\em }}
 }

\maketitle

\begin{abstract}
We consider the bond percolation model on the lattice $\Z^d$ ($d\ge 2$) with the constraint to be fully connected. Each edge is open with probability $p\in(0,1)$, closed with probability $1-p$ and then the process is conditioned to have a unique open connected component (bounded or unbounded). The model is defined on $\Z^d$ by passing to the limit for a sequence of finite volume models with general boundary conditions. Several questions and problems are investigated: existence, uniqueness, phase transition, DLR equations. Our main result involves the existence of a threshold $0<p^*(d)<1$ such that any infinite volume model is necessary the vacuum state in subcritical regime (no open edges) and is non trivial in the supercritical regime (existence of a stationary unbounded connected cluster). Bounds for $p^*(d)$ are given and show that it is drastically smaller than the standard bond percolation threshold in $\Z^d$. For instance $0.128<p^*(2)<0.202$ (rigorous bounds) whereas the 2D bond percolation threshold is equal to $1/2$.

%

\end{abstract}

{\it key words:} FK-percolation, random cluster model, phase transition, FKG inequalities, DLR equations.

\section{Introduction}

In the standard bond percolation model on $\Z^d$, all edges are independently open with probability $p\in(0,1)$ and closed with probability $1-p$. Then the main questions involves the existence or not, the size, the shape of the unbounded connected component of open edges with respect to the parameter $p$. This model is abundantly studied in the percolation theory literature \cite{DC18,Grimmett}. In the present paper we consider this bond percolation model on $\Z^d$ with the constraint to be fully connected. Heuristically it means that we condition the Bernoulli bond percolation model by the event claiming that the number of open  connected components (bounded and unbounded) is equal to one. Obviously this conditioning is forbidden since that the probability measure of the event is null. To this end two rigorous ways are possible and explored here. The first one is to consider the model in a finite window (with some boundary conditions) and to pass to limit with respect to the size of the window.  A second approach is to consider the existence and the description of the model via the so-called DLR (Dobrushin-Lanford-Ruelle) equations which prescribe the local conditional distributions via specifications (see equations \eqref{DLRintro} for details). In the following for $p\in(0,1)$ we denote by $\LL_s(p)$ all stationary accumulation points for the limiting procedure and by $\G_s(p)$ the set of all stationary Gibbs measures, solutions of DLR equations.

Our main result involves the existence of a non trivial threshold $0<p^*(d)<1$ which depends only on the dimension $d\ge 2$ such that 

\begin{itemize} 
\item (subcritical regime) for $p<p^*(d)$, any thermodynamic limit does not have open edges. All edges are closed. The set $\LL_s(p)$ is reduced to the probability measure charging the null configuration (the vacuum state).  Equivalently the set $\G_s(p)$ is empty.

\item (supercritical regime) for $p>p^*(d)$, there exists a non trivial stationary distribution $P$ belonging to $\LL_s(p)$ and $\G_s(p)$.

\end{itemize} 

Note that in the supercritical regime the uniqueness of elements in $\LL_s(p)$ or $\G_s(p)$ is not guaranteed in general. A uniqueness result is only given in the case $d=2$ and $p\ge 1/2$. Rigorous bounds for $p^*(d)$ are also provided (see Theorem \ref{theoboundsintro}). In particular in dimension $d=2$ we have $0.128<p^*(2)<0.202$ and it is remarkable to see that it is drastically  smaller than the 2D bond percolation threshold equals to $1/2$. In any dimension $d\ge 2$, $p^*(d)$ is smaller than the standard percolation threshold $p_c(d)$ and the lower and upper bounds are sharp enough to imply that $p^*(d)\sim e^{-1} p_c(d)$ when $d\mapsto +\infty$.

Further its own interest, there are several motivations to study the fully connected bond percolation model. First this model is related to the so-called random cluster model or FK-percolation; see \cite{RCM} for a general presentation on lattice and \cite{DH} for a recent  version in the continuum. It is defined via the formal unnormalised density $q^{N_{cc}}$ where  $N_{cc}$ is the number of connected component and $q>0$  a positive parameter. Our setting corresponds to the random cluster model associated to the  Widom-Rowlinson lattice model with $q\to 0$ \cite{GMH}. The case $q<1$ is less studied in the literature since the central and crucial FKG inequalities are lost. Our first motivation was to study the case "$q=0$"  without FKG inequalities and monotonicity and to prove a sharp phase transition phenomena as presented above. Note that phase transition results without FKG inequalities/monotonicity are rare in the literature and it still remains a general and global challenge  for several models in  statistical physics.

Connexions with the incipient cluster at criticality are also possible. In dimension $d=2$ it is well known that the Bernoulli bond model does not percolate at criticality $p=p_c(2)=1/2$. The incipient cluster has been introduced to force percolation at criticality by conditioning to the existence of a connected component from  $0$ to infinity. As in our setting  the conditioning is not possible and several strategy has been developed to give a sense to the conditional probability \cite{ BasuIncipient,Antal,IncipientKesten}. With our conditioning we force percolation at criticality and
partially at subcriticality as well.

The main original idea developed in the present paper is to merge the fully connected percolation model in a more general model of statistical physics on $\Z^d$ with two parameters $(\lambda,\mu)\in\R^2$. The formal Hamiltonian is given by 

$$ H=\lambda N+ \mu\partial N +\infty \1_{\{N_{cc}\neq 1\}},$$

where $N$ is the number of open edges, $\partial N$ is the number of closed edges sharing at least one vertex with an open edge and $N_{cc}$ is the number of open clusters. The fully connected bond percolation model corresponds to the case $\lambda=\log(p/(1-p))$ and $\mu=0$. Using tools of statistical physics we study the two-parameter model and prove for any $\mu\in\R$ a phase transition phenomenon with respect to $\lambda$. The tools are related to the pressure function and its properties (convexity, differentiability, non-dependence of boundary conditions, etc). The sharpness of bounds for $p^*(d)$ are also due to the connexion with this two-parameter model. 

The two-parameter model  has also its own interest since it is a special case of random connected weighted sub-graphs picked uniformly in a host graph. We find these kind of graphs in several domains of applied science \cite{LuBressan}. Here the weights are simple and encoded by only two quantities (the volume and the perimeter of the connected component) but more complicated weights and other host graphs than $\Z^d$ could be considered. We believe that our proof of phase transition phenomenon is robust enough to be applied for many different Hamiltonians. Note that the phase transition "empty configuration/unbounded configuration" is relevant for application since it corresponds to the emergence of a macroscopic object inside a very large host graph.

%
%
%

%
%
%
%

Let us finish the introduction with a numerical illustration of the phase transition phenomenon. Using a birth-death Metropolis Hastings algorithm, we sample the fully-connected bond percolation model with free boundary condition on a 2D grid $30*30$. The simulation highlights that  $p^*(2)$ is between 0.15 and 0.2. At the middle, a simulation of the process with $p=0.2$. The monitoring control on the left gives the number of open edges during the run of the algorithm and shows that the equilibrium state seems to be attained. The monitoring control on the right is for $p=0.15$ and shows that the connected component has a microscopic size with respect to the size of the window. It  would disappear at the limit when the size of the window tends to infinity

\begin{figure}[htbp]
    \setlength{\tabcolsep}{0.1cm} \centerline{
      \begin{tabular}[]{cccc}
  \includegraphics[angle=0,scale=.22]{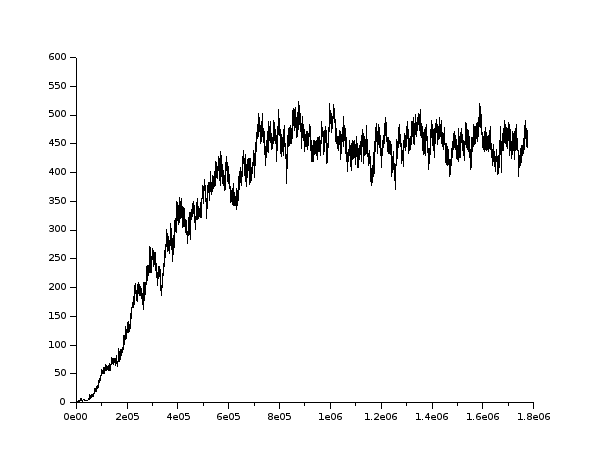} &
  \includegraphics[angle=0,scale=.27]{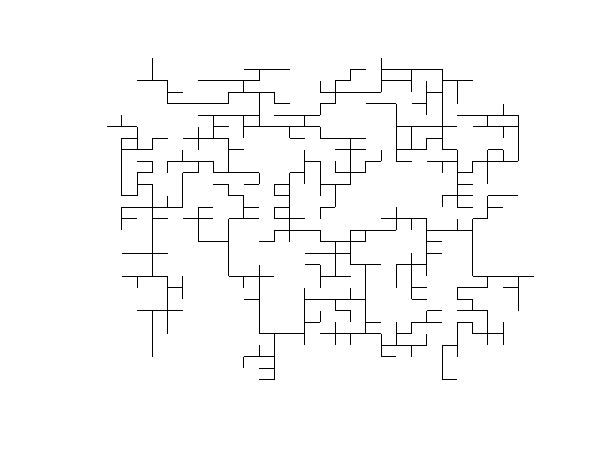} &
       \includegraphics[angle=0,scale=.22]{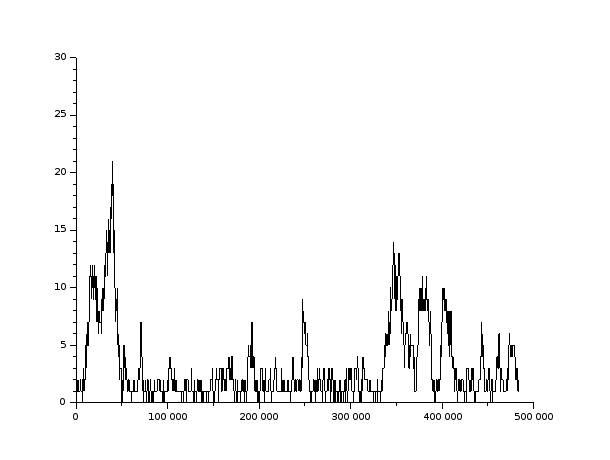} 
      \end{tabular}      
  }
  \end{figure}

The plan of paper is the following. In Section \ref{Section_model} we present the fully connected bond percolation model. The results are given in Section \ref{Section_Results} and the main tools and ideas in Section \ref{Section_Tools}. As mentioned above the fully connected bond percolation model is merged in a more general model with two parameters. It is investigated in details (results and proofs) in Section \ref{Section_2parModel}. In the last Section \ref{Section_proofs}, we give the proofs of results presented in Section \ref{Section_Results}; they are partially based on results from Section \ref{Section_2parModel}. 

\tableofcontents

\section{The fully connected bond percolation model}\label{Section_model}

\subsection{Description of the model}

We denote by $\EE$ the set of edges in $\Z^d$, where an edge is a couple of vertices with distance one for the $L^1$ norm.  The  space of configurations $\O$ is defined by $\{0,1\}^\EE$ and is is equipped with the standard $\sigma$-algebra generated by cylinders. For any $\o$ in $\O$, we say that an edge $e\in\EE$ is open if $\o(e)=1$; it is closed otherwise ($\o(e)=0$). For any $A\subset \EE$ we denote by $\o_{A}$ the restriction of $\o$ to $A$ (i.e. an element in $\{0,1\}^\A$) and for $\o,\o'$ two configurations in $\O$, we denote by $\o_A\o'_{A^c}$ the concatenation of configurations $\o_A$ and $\o'_{A^c}$. For any $\L\subset \Z^d$, the set $\EE_\L$ denotes the edges in $\EE$ such that both extremities are in $\L$.  The space of configuration $\O_{\L}$  is defined by $\{0,1\}^{\EE_\L}$. In the following, without any ambiguity,   we denote by $\o_\L$ an element in $\O_{\L}$ or the restriction of an element $\o\in\O$ to $\O_\L$ (i.e. $\o_\L=\o_{\EE_\L}$). 

For any $p\in (0,1)$, we denote by $\P_p$ (respectively $\P_p^\L$)  the probability measure $B(p)^{\otimes \EE}$ (respectively $\B(p)^{\otimes \EE_\L}$) on $\O$ (respectively $\O_\L$) where $\B(p)$ is the standard Bernoulli distribution with parameter $p$. The set of allowed configurations is denoted by $\A$;
$$\A=\Big\{\text{There is only one connected component of open edges }\Big\}.$$

We have to give a sense to "$\P_p(.|\A)$":\\

{\it -The thermodynamic approach.} We consider the window $\L_n=\{-n,n\}^d$ and the probability measure $\P^{\L_n}_p(. |\A)$ which is well defined. Since that the state space $\Omega$ is compact the sequence $(\P^{\L_n}_p(. |\A))_{n\ge 1}$ admits accumulation points for the weak convergence of measures. Boundary conditions can also be considered in a general matter (wired, free, periodic); for the sake of simplicity the details involving the boundary conditions are given in Definition \ref{definition_boundary} below. We denote by $\LL(p)$ all possible accumulation points for all possible boundary conditions. We are mainly interested in probability measures $P\in\LL(p)$ which are invariant by translations (stationary in space). This set is denoted by $\LL_s(p)$.\\

{\it -The DLR equations approach.} We are looking for  probability measures $P$ on $\O$ such that $P(\A)=1$ and such that given the outside configuration $\o_{\L^c}$ the configuration $\o_\L$ inside $\L$ is picked uniformly on $\A$ with the distribution $\P_p^\L$. That leads to the following conditional distribution for $P$-a.e. $\o_{{\EE^c_\L}}$

\begin{equation}\label{DLRintro}
P(d\o_\L|\o_{{\EE^c_\L}})=\frac{1}{Z_\L(\o_{{\EE^c_\L}})} 1_{\A}(\o) \P_p^\L(d\o_\L), 
\end{equation}

where $Z_\L(\o_{{\EE^c_\L}})$ is the normalizing constant. The collection of equations \eqref{DLRintro} for all bounded $\Lambda\subset \Z^d$ are called DLR equations (for Dobrushin, Lanford and Ruelle). We denote by $\G(p)$ the set of all probability measures $P$ such that $P(\A)=1$ and such that DLR equations hold. The set of probability measures $P\in\G(p)$ which are invariant by translations is denoted  $\G_s(p)$. 

Both approaches are common in statistical physics. For a non-expert reader we advice the general references \cite{Velenik,GeorgiiBook}. Note that at this stage only the existence of element $P\in\LL(p)$ is ensured by a simple compactness argument. The set $\LL_s(p)$, $\G(p)$ and $\G_s(p)$ could be empty. We will introduce a periodic boundary condition which necessary will produce stationary elements in  $\LL(p)$. So $\LL_s(p)$ will be not empty as well.

\subsection{Results}\label{Section_Results}

Our main motivation is to study the elements and the geometry  of $\LL(p), \LL_s(p), \G(p)$ or $\G_s(p)$ with respect to the parameter $p\in(0,1)$. Can we observe phase transition phenomena? i.e. abrupt modifications for some special critical values for $p$? Are the spaces identical and/or reduced to a singleton and/or empty? 

A first natural question is to know if  $\LL(p)\neq\LL_s(p)$ or $\G(p)\neq\G_s(p)$. This phenomenon is well-known and identified by the statistical physics community as the symmetry breaking (here the translations would be broken). It is a general and difficult question mainly open for all models. We do not investigate it here and we will focus mainly (but not only) on $\LL_s(p)$ and  $\G_s(p)$.

Our first result involves the connectivity of elements  $P$ in $\LL_s(p)$ or $\G_s(p)$. In ad-equation with our heuristic definition we expect that $P(\A)=1$. It is guaranteed by definition for $P\in\G_s(p)$ but for $P\in\LL_s(p)$ it is more complicated because during the thermodynamic limit two bad phenomena may occur: the unique connected component can disappear at infinity and it remains no open edge; i.e. the probability measure limit is the measure $\delta_{0^\EE}$, where

$$ 0^\EE \text{ is the configuration full of } 0 \text{ (the vacuum configuration)}.$$

or the unique connected component in finite volume regime splits in severals parts in the infinite volume regime. This second scenario is in fact impossible.

\begin{theo}\label{thegeointro}
For any $P\in\LL_s(p)$, $P(\A\cup\{0^\EE\})=1$.
\end{theo}

Let us  give an heuristic argument explaining why the limit can be $0^\EE$. If $p$ is small the unique connected component produced by  $\P^{\L_n}_p(. |\A)$ has a microscopic size with respect to the size of the window and disappears at the limit. Actually $p$ has to be large enough in order to produce a macroscopic connected component which survives when passing to the limit. It is natural to look for a threshold which separates the existence/non-existence of a macroscopic connected component in the thermodynamic limit. However the existence of such a threshold is not obvious since the model does not inherit any stochastic monotony with respect to $p$ (FKG inequality does not hold here).

\begin{theo}\label{theothtesholdintro}
For any $d\ge 2$, there exists a threshold $0<p^*(d)<1$ such that

\begin{itemize}
\item if $p>p^*(d)$,  there exists $P$ in $\in\LL_s(p)\cap \G_s(p)$ with $P(\A)=1$.
\item if $p<p^*(d)$, $\G_s(p)=\emptyset$ and $\LL_s(p)=\{\delta_{0^\EE}\}$.
\end{itemize}
\end{theo}

Moreover we can prove that $\G_s(p)$ and $\LL_s(p)$ are almost identical excepted the possible existence of the null configuration. Actually $ \G_s(p)\subset \LL_s(p)$ and for all $ P\in\LL_s(p)$ such that $P(\A)>0$ then $P(.|\A)\in\G_s(p)$. 

Note that for $p>p^* (d)$ the unique connected component is necessary unbounded (by a stationary argument) and so percolation occurs. An experimented reader in percolation theory should be interested in comparing the threshold $p^*(d)$ in Theorem \ref{theothtesholdintro} with the standard percolation threshold $0<p_c(d)<1$ defined by
\begin{equation}\label{thresholdperco}
 p_c(d)=\inf\{p>0,\,\P_p(\text{ there exists an unbounded connected component})=1\}.
 \end{equation}
We advice \cite{Grimmett} for definition and  first properties of such a threshold. In particular in the following we will use the facts that $p_c(2)=1/2$ and $p_c(d)\sim 1/2d$ when $d\to \infty$.   Our main and original result provides sharp bounds for $p^*(d)$ and comparison with $p_c(d)$.

\begin{theo}\label{theoboundsintro}
For any $d\ge 2$
 
\begin{equation}\label{boundsintro} 
 \frac{e^{\lambda_\text{min}^*(d)}}{1+e^{\lambda_\text{min}^*(d)}}\le p^*(d) \le \frac{e^{\lambda_\text{max}^*(d)}}{1+e^{\lambda_\text{max}^*(d)}},
 \end{equation}
  with 
 $$ \lambda_\text{min}^*(d)=-\log(2d-1) + (2d-2)\log\left(\frac{2d-2}{2d-1}\right),$$
 $$ \lambda_\text{max}^*(d)=-\log(p_c(d)) + \frac{1-p_c(d)}{p_c(d)} \log(1-p_c(d)),$$
\end{theo}

In particular, for $d=2$, since $p_c(2)=1/2$ we obtain that  

\begin{equation}\label{encadrement1} 0.128<p^*(2)<0.202.
\end{equation}

In the case $d=3$, the threshold $p_c(3)$ is unknown but numerical approximation gives $p_c(3)\approx 0.25$ \cite{WZZTD} and therefore

\begin{equation}\label{encadrement2}
0.075<p^*(3)<0.099.
\end{equation}

It is remarkable to note that $p^*(d)$ is drastically smaller than $p_c(d)$. Actually a simple study of bounds in Theorem \ref{theoboundsintro} provides the following asymptotic result.

\begin{coro}
When $d\to \infty$ the following equivalence holds
$$ p^*(d)\sim e^{-1}p_c(d). $$ 
\end{coro} 

We do not have any reason to believe that $ p^*(d)$ would be equal to $ e^{-1}p_c(d) $ for all $d\ge 2$. However note that identities $e^{-1} p_c(2)=0.184$ and $e^{-1}p_c(3)\approx 0.091$  are compatible with bounds \eqref{encadrement1} and \eqref{encadrement2}.

Let us now turn to a last theorem which claims that for $p$ large enough the four sets $\LL(p), \LL_s(p), \G(p)$ and $\G_s(p)$ are identical and reduced to a singleton (modulo a vacuum part).

\begin{theo}\label{theouniciteintro} For $d=2$ and $p\ge 1/2$ there exists a stationary probability measure $P$ such that

$$\G(p)=\G_s(p)=\{P\}.$$ 

Moreover for any $Q\in\LL(p)$, there exists $\alpha\in[0,1]$ such that $Q=\alpha P+(1-\alpha) \delta_{0^\EE}.$

\end{theo}

Note that the theorem holds for $p=1/2$ corresponding to the critical setting in dimension $d=2$. So, as for the incipient cluster \cite{IncipientKesten}, the forced unbounded connected component process is unique in distribution at criticality. In dimension $d\ge 3$ the theorem remains true for $p$ large enough but our proof would be valid only for $p$ strictly larger than $p_c(d)$. We omit to give it. We believe that mainly $\alpha=1$ in Theorem \ref{theouniciteintro} but we did not succeed to remove the possible vacuum part in general. However the vacuum part could exist for $p=p^*(d)$ corresponding to a possible liquid-gas phase transition.

\subsection{Main tools and ideas}\label{Section_Tools}

Let us present the main tools and ideas we use to prove theorems presented above. Theorem \ref{thegeointro} is based on two steps. First we show that the connected components of any accumulation point $P\in\LL(p)$ are necessary unbounded. Then a general Burton and Keane strategy ensures that the number of unbounded connected components is equal to zero or one. 

The existence of the threshold in Theorem \ref{theothtesholdintro} is more delicate. Recall that this model does not exhibit any natural monotonicity with respect to $p$. Actually Theorem \ref{theothtesholdintro} is based on a precise analysis of the function  $\PP: (0,1)\to\R$ defined by

$$ \PP(p)=\lim_{n\to \infty} \frac{1}{\#(\EE_n)} \log(\P_p^{\L_n}(\A)).$$
  
Excepted scaling constants, this function is called pressure by the statistical physics community. It is common to identify phase transition phenomenon by detecting lost of analyticity of $\PP$ for special critical values of $p\in(0,1)$. We use this strategy here. In fact $\PP$ is larger or equal to $-\log(2)$ and we show that  that the expected threshold $p^*(d)$ is defined by

\begin{equation}\label{defthresholintro}
p^*(d)= \inf(p\in(0,1), \PP(p)>-\log(2)).
  \end{equation}
  
  Combining different tools from statistical physics we show Theorem \ref{theothtesholdintro}.  

Theorem \ref{theouniciteintro} on the uniqueness of accumulation points or Gibbs measures is obtained by a coupling algorithm. This result is strongly inspired by the disagreement percolation argument to prove uniqueness of Gibbs measures \cite{VDBM}.

  The main original part of our work involves  Theorem \ref{theoboundsintro}. First, using simple combinatorial arguments it is possible to obtain trivial bounds for $ p^*(d)$. However they are not sharp as they are in Theorem \ref{theoboundsintro}. Our strategy is to merge the model with one parameter $p\in(0,1)$ inside a model with two parameters $(\lambda,\mu)\in\R^2$. The finite volume definition on $\Lambda\subset \Z^d$ is given by the following distribution

\begin{equation}\label{defQnintro}
Q^\L_{\lambda,\mu}(\o_\L):=\frac{1}{Z_\L(\lambda,\mu)} \1_{\A}(\o_\L) e^{\lambda N_\L(\o_\L)} e^{\mu \partial N_\L(\o_\L)},\quad \o_\L\in\O_\L,
\end{equation}

where $Z_\L(\lambda,\mu)$ is the normalising constant,  $ N_\L(\o_\L)$ is the  number of open edges in $\o_\L$ and $\partial N_\L(\o_\L)$ is the number of closed edges such that at least one of its extremities belongs to an open edge of $\o_\L$. The quantity  $ N(\o_\L)$ can be viewed as the size of the open cluster whereas $\partial N(\o_\L)$ is the perimeter of the open cluster. It is easy to see that $\P_p^\L(.|\A)$ corresponds to $Q^\L_{\lambda,\mu}$ with $\lambda=\log(p/(1-p))$ and $\mu=0$ (modulo a boundary effect). We show that for any $\mu\in\R$ there exists a critical parameter $\lambda^*(\mu)\in\R$ as in Theorem \ref{theothtesholdintro} and therefore

$$ p^*(d)= \frac{e^{\lambda^*(0)}}{1+e^{\lambda^*(0)}},$$

which explains the special form of bounds in \eqref{boundsintro}. The interest of this two-parameter model comes from its partial tractability since for $\mu\le \log(1-p_c(d))$ we have explicitly 

$$ \lambda^*(\mu)=\log(1-e^\mu).$$

It is exactly at this point that the standard percolation threshold $p_c(d)$ appears. Now exploiting the convexity of the pressure of the two-parameter model and some estimates on its derivatives we obtain sharp bounds for $\lambda^*(0)$ and the proof of Theorem \ref{theoboundsintro} follows.

\section{The two-parameter model}\label{Section_2parModel}

\subsection{Connected components and boundary conditions}
%
%

For a configuration $\o\in\O$, we denote by $N_{cc}(\o)$ the number of connected components of open edges. Before defining the number of connected components for an element $\o_\L\in\O_\L$, we introduce the notion of boundary condition in a general matter. We denote by $\partial \Lambda$ the boundary of $\Lambda$ (i.e. any vertex $i\in\Lambda$ such that $i$ has a neighbour in $\Lambda^c$).

\begin{defn}\label{definition_boundary}
For any bounded subset $\Lambda$, a boundary condition of $\L$ (noted "$bc(\L)$"  or simply "bc" if no  ambiguity holds) is any partition of 
$\partial \Lambda$ where each set of the partition is marked by $0$ or $1$; i.e. any collection $((E_1,\delta_1),(E_2,\delta_2),\ldots,(E_k,\delta_k))$ such that $(E_1,E_2,\ldots,E_k)$ is a partition of $\partial \Lambda$ and $\delta_i\in\{0,1\}$, $1\le i\le k$. A marked set $(E_i,\delta_i)$ corresponds to a collection of vertices which are identified to be a single point (by considering for instance the quotient space) which is closed or open depending on the value of $\delta_i$ (as usual, $1=$"open" and $0=$"closed").

The free boundary condition is a partition with closed singletons; i.e. $((\{i\},0), i\in\partial \L)$. The wired boundary condition is a partition with only one open set $\partial \Lambda$; i.e. $((\partial \Lambda,1))$. The left-right crossing boundary condition is a partition with two open sets (the right and left faces of a cube $\L$), the other sets are closed singleton. The periodic condition is a partition with closed pairs of opposite vertices at the boundary. Several other combinations are possible as "closed wired", "open periodic", etc...
%
\end{defn}

Then for any $\o_\L$ in $\O_\L$, we denote by $N_{cc}^{bc(\L)}(\o_{\L})$ the number of connected components in $\o_{\L}$ taking account the boundary condition "$bc(\L)$". To avoid any confusion, it is defined as the number of connected components in the following graph:

\begin{itemize}
\item  the vertices are the points in $\L\backslash \partial \L$, which belongs to an open edge in $\o_\L$, and any set $E$ of the partition $bc(\L)$ which is either open or if at least one vertex of $E$ belongs to an open edge in $\o_\L$.
\item the edges are induced by open edges in $\o_\L$. 
\end{itemize}

 A natural way for producing a boundary condition of $\L$ is to use the open edges of a configuration $\o$ outside $\L$. Precisely, for any $\o\in\O$ and $\L\subset \Z^d$, we denote by $bc(\L,\o)$ the following boundary condition : the points at the boundary are in the same set $E$ of the partition if they are connected by a path of open edges in $\o$ not belonging in $\EE_\L$ (they are connected from outside of $\L$). Each set $E$ of the partition is declared open if at least one vertex of $E$ belongs to an open edge in $\o$ not belonging in $\EE_\L$ (it is open from outside of $\L$).

Recall that the space of allowed configurations in $\O$ is $\A=\{\Ncc=1\}$. Similarly $\A^{bc(\L)}$ is the space of configurations $\{N_{cc}^{bc(\L)}=1\}$ in $\O_\L$ for the boundary condition $bc(\L)$.

In the following we often use the bounded box  $\L_n=\{-n,\ldots, n\}^d\subset\Z^d$ for $n\ge 1$. In this case, we use the notations $\O_n$, $\EE_n$, $N_n$, $bc(n)$, $N^{bc(n)}_{cc}(\o_{\L_n})$ and $\A^{bc(n)}$ in place of $\O_{\L_n}$, $N_{\L_n}$,  $\EE_{\L_n}$, $bc(\L_n)$, $N^{bc(\L_n)}_{cc}(\o_{\L_n})$ and $\A^{bc(\L_n)}$.

\subsection{Finite volume models}
%
%

 For any $p\in (0,1)$ and $\Lambda\subset \Z^d$, recall that $\P_p^\L$ denotes the probability measure $B(p)^{\otimes \EE_\L}$ on $\O_\L$ and so for every $\o_\L$

\begin{equation}\label{Bernoulli}
\P^\L_p(\o_\L) =c_\L \left(\frac{p}{1-p}\right)^{N_\L(\o_\L)} =c_\L e^{\lambda N_\L(\o_\L)},
\end{equation} 
where $c_\L:=1/(1-p)^{\#\EE_\L}$ is the normalization constant,   $N_\L(\o_\L)$ is the number of open edges in $\o_\L$ and $\lambda$ is the parameter $\log(p/(1-p))$.

Before introducing the finite volume model, we define a last quantity providing the number of closed edges with open neighbour. For any $\L\subset \Z^d$, any $\o_\L$ in $\O_\L$ and any boundary condition $\bc(\L)$ we denote by $\partial N_\L^{\bc(\L)}(\o_\L)$ (or simply $\partial N_\L^{\bc}(\o_\L)$ to avoid repetition) the number of closed edges in $\o_\L$  such that at least one of its extremities belongs to an open edge of $\o_\L$ or an open set $(E,1)$ in the boundary condition $\bc(\L)$.

\begin{defn}\label{defintionQ} Let $\L$ be a connected bounded subset of $\Z^d$, "$\bc(\L)$" be a boundary condition and $\lambda,\mu$ be two real numbers in $\R$. The fully-connected bond measure on $\L$ with parameter $(\lambda,\mu)$ and boundary condition "$bc(\L)$" is  the probability measure on $\O_\L$ defined by
\begin{equation}\label{defQn}
Q_{\L,\lambda,\mu}^{\bc}(\o_\L):=\frac{1}{Z_\L^{\bc}(\lambda,\mu)} \1_{\left\{\Ncc^{\bc(\L)}(\omega_\L)=1\right\}} e^{\lambda N_\L(\o_\L)} e^{\mu \partial N^\bc_\L(\o_\L)},
\end{equation}
where $Z_\L^{\bc}(\lambda,\mu)$ is the normalization constant. Note that  $Z_\L^{\bc}(\lambda,\mu)\ge e^{\lambda\#\EE_\L}>0$ since the configuration  with all open edges is allowed.
\end{defn}

With a good choice of parameters we identify two main models.

{\bf -Example 1 (infinite cluster of Bernoulli bond percolation):  }
For any $p\in (0,1)$, we fix $\lambda=\log(p)$ and $\mu=\log(1-p)$. By a simple identification we see that the weight 

$$ \1_{\left\{\Ncc^{\text{wired}}(\omega_\L)=1\right\}} p^{N_\L(\o_\L)}(1-p)^{\partial N^\text{wired}_\L(\o_\L)}$$

of $Q_{\L,\lambda,\mu}^{\text{wired}}$ is exactly the probability that the open edges of $\o_\L$ corresponds to open clusters hitting the boundary of $\L$ for a Bernoulli bond percolation $\P_p^\L$. In particular $Z_\L^{\text{wired}}$ is a sum of probability weights and is equal to one. In the thermodynamic limit 
(i.e. $\L\to \Z^d$), $Q_{\L,\lambda,\mu}^{\text{wired}}$ converges to the distribution of the infinite cluster in $\P_p$ if it exists. If it does not exist $Q_{\L,\lambda,\mu}^{\text{wired}}$ converges to the null configuration $0^\EE$ (all edges are closed). See Proposition \ref{propbernoulli} for details..

{\bf  -Example 2 (Fully connected bond percolation model):  }
For any $p\in (0,1)$, we fix $\lambda=\log(p/(1-p))$ and $\mu=0$. Then the probability measure $Q_{\L,\lambda,0}^{\text{bc}}$ is simply the distribution of a Bernoulli bond model with parameter $p$ conditionned to have a unique connected component (with respect to the boundary condition bc);

$$ Q_{\L,\lambda,0}^{\text{bc}} = \P^\L_p(.|\Ncc^{\text{bc}}=1).$$

In particular for $p=1/2$, $Q_{\L,0,0}^{\text{free}}$ samples randomly and uniformly a connected graph in $\L$. 
 
%
%

\subsection{Thermodynamic limits} 

Let us now turn to the main object of the present paper; any thermodynamic limit of finite volume fully connected bond models. For simplicity, in the following the thermodynamic limits are always along the sequence of boxes $(\L_n)_{n\ge 1}$ and the boundary condition "\bc" denotes in fact a sequence of boundary conditions "$(\bc(n))_{n\ge 1}$".  The probability measures $Q_{\L_n,\lambda,\mu}^{\bc(n)}$ is simply denoted by $Q_n^{bc}$.

\begin{defn} We denote by $\LL^{bc}(\lambda,\mu)$ the set of all accumulation points of  
$(Q_n^{bc})_{n\ge 1}$ (embedded in $\O$) for the weak convergence of measures. Any element in $\LL^{bc}(\lambda,\mu)$ is called a  fully-connected bond measure with parameters $(\lambda,\mu)$ and boundary condition "bc". $\LL(\lambda,\mu)$ is the union of all $\LL^{bc}(\lambda,\mu)$ for any choice of boundary condition "bc". We denote by  $\LL_s^{bc}(\lambda,\mu)$ (or  $\LL_s(\lambda,\mu)$) the elements of  $\LL^{bc}(\lambda,\mu)$ (or  $\LL(\lambda,\mu)$) which are stationary in space, meaning invariant in distribution with respect to any translation  $\tau_x$ by vector $x\in\Z^d$.
\end{defn}

Since $\O$ is compact the space  $\LL^{bc}(\lambda,\mu)$ is not empty for any $\lambda$, $\mu$ and sequence of boundary conditions "$(\bc(n))_{n\ge 1}$". The existence of elements in 
$\LL_s^{bc}(\lambda,\mu)$  for any "bc" is more delicate and discussed below. However for the periodic boundary condition "bc=per" the accumulations points are necessary stationary in space and so $\LL^{per}(\lambda,\mu)$ and $\LL_s(\lambda,\mu)$ are not empty. The following proposition provides a clear interpretation of the thermodynamic limits in the case of example 1 above.

We denote by $\P_p^\infty$ the probability on $\O$ such that the open edges are distributed as the unique infinite open cluster in the Benoulli percolation $\P_p$, if percolation occurs. If percolation does not occur, $\P_p^\infty$ is not defined. So $\P_p^\infty$ is defined for $p>p_c(d)$ and for $p=p_c(d)$ if the conjecture "$\theta(p_c)=0$" is not true. Note that $\P_p^\infty(\A)=1$.

\begin{prop}\label{propbernoulli}
For any $p\in (0,1)$, we fix $\lambda=\log(p)$ and $\mu=\log(1-p)$. 

\begin{itemize}
\item If $p>p_c(d)$, then $\LL^{\text{wired}}(\lambda,\mu)=\{ \P_p^\infty\}$ 
\item If $p<p_c(d)$, then $ \LL^{\text{wired}}(\lambda,\mu)=\{ \delta_{0^\EE}\}$.
\item If $p=p_c(d)$, it depends if the infinite cluster exists at criticality (conjecture "$\theta(p_c)=0$"). If it exists  then $\LL^{\text{wired}}(\lambda,\mu)=\{ \P_p^\infty\}$ otherwise  $\LL^{\text{wired}}(\lambda,\mu)=\{ \delta_{0^\EE}\}$.
\end{itemize}

\end{prop}

\begin{proof}
The proof is based on the description of $Q_{\L_n,\lambda,\mu}^{\text{wired}}$ given in example 1 which is identified as the distribution of open clusters hitting the boundary of $\L_n$ for a Bernoulli bond percolation $\P_p^{\L_n}$. So the proof of the proposition is a simple consequence of the following coupling. For any $n\ge 1$ and $\o\in\O$ we consider the configuration $\o_n \in \O_{\L_n}$ whom open edges are the open clusters of $\o$ inside $\L_n$ hitting the boundary of $\L_n$. Therefore if $\o$ is distributed with respect $\P_p$ then $\o_n$ is distributed with respect to $Q_{\L_n,\lambda,\mu}^{\text{wired}}$. It is now clear that the local limit of $\o_n$,  when $n\to\infty$, is the infinite open cluster of $\o$ if it exists or the vacuum configuration $0^\EE$ if it does not exist. We deduce that the weak limit of $Q_{\L_n,\lambda,\mu}^{\text{wired}}$ when $n\to\infty$ is the distribution $\P_p^\infty$ if percolation occurs or $\delta_{0^\EE}$ if it does not.

\end{proof}

%
%
%
%
%
%

\subsection{Connectivity properties}

In this section we investigate the topological properties of elements $P$ in $\LL(\lambda,\mu)$.

\begin{prop}\label{Prop_notbounded} For any $P$ in $\LL(\lambda,\mu)$, $P$-almost surely the connected components of open edges are unbounded.
\end{prop}

\begin{proof}
Let $P$ in $\LL(\lambda,\mu)$. First we show that for $P$-almost every $\omega$, if $\omega$ contains a bounded connected component then $\omega$ is reduced to this bounded component (in a second step we show that it is impossible). We make a proof by contradiction by assuming the opposite: with positive probability there exist a bounded connected component and another connected component (bounded or not). So we can find a bounded set $\L\subset \Z^d$ large enough such that
$$P\left(\begin{array}{l}
\text{-There exists a connected component completly incuded in } \L\\
\text{-There exists another connected component hitting } \L.\end{array}
\right)>0.$$ 
But this event is local with probability zero for any $Q^{bc}_n$. That implies a contradiction by weak convergence of  $(Q^{bc}_n)$ to $P$ (for a subsequence).

It remains to prove that a unique bounded connected component can not occur with positive probability. We make again a proof by contradiction. Assume that there exists a bounded set $\L\subset \Z^d$ such that 
$$P\left(\begin{array}{ll}
\text{There exists a unique bounded connected component} \\
\text{and it is incuded in } \L
\end{array}
\right)>0.$$
  By the weak convergence of $(Q^{bc}_n)$ to $P$ (for a subsequence) and the previous result, we deduce that there exist $\delta>0$ and $n_0\ge 1$ such that for any $n\ge n_0$
$$Q_n^{\bc}(\text{There exists a unique bounded connected component and it is incuded in } \L)>\delta.$$
If the boundary condition has an open vertex, this implies a contradiction for $n$ large enough. If the boundary condition is completely closed, we choose an integer $k$ large enough such that $k\delta>1$. We choose also $n$ large enough such that $\L_n$ contains $k$ disjoint copies $(\L^i)_{1\le i\le k}$ obtained by translations of  $\L$. So by definition of $Q_n^{\bc}$, we obtain for $1\le i\le k$, 
$$Q_n^{\bc} \left(\begin{array}{ll}
\text{There exists a unique bounded connected component } \\
\text{ and it is incuded in } \L^i
\end{array}\right)>\delta.$$
But these $k$ events are disjoint and $k\delta>1$. It is impossible.
\end{proof}

Now we investigate the number of unbounded connected components. To this end, a precious tool is the finite energy property and the  general Burton and Kean argument. The model here does not satisfy exactly the finite energy property but only the following variant.

\begin{lemm}\label{localmodif}
Let $\L$ be a bounded connected set in $\Z^d$. Let $P\in\LL(\lambda,\mu)$ and let $A$ be an event in $\O$ measurable with respect to the sigma-field generated by events $(\{e \text{ is open}\})_{e\in  \EE^c_{\L}}$. We assume that

$$ P(A\cap\{\text{an edge between a vertice in } \L \text{ and } \L^c \text{ is open}\})>0.$$  

Then 
$$P (A\cap\{\text{ all edges in } \EE_\L \text{  are open}\})>0.$$
\end{lemm}

\begin{proof}
Let $P$ be in $\LL(\lambda,\mu)$. For sake a simplicity, we note $P=\lim_{n\to \infty} Q_n^{bc}$ in omitting the limit under a subsequence.

By the martingale convergence Theorem, the indicator function $\1_A$ can be approximated by local bounded functions which are measurable with respect to the sigma-field generated on  $\EE^c_{\L}$. Indeed, $P$ almost surely $\1_A=\lim_{\Delta\to \Z^d} E(\1_A |\mathcal{F}^\L_\Delta)$ where $\mathcal{F}^\L_\Delta$ is the sigma-field generated on  $\EE_\Delta\backslash \EE_{\L}$. Then

\begin{eqnarray*}
& & P (A\cap\{\text{ all edges in } \EE_\L \text{  are open}\})\\
&=& \lim_{\Delta\to \Z^d} E_P\left( E(\1_A |\mathcal{F}^\L_\Delta) \1_{\{\text{ all edges in } \EE_\L \text{  are open}\}}\right)\\
&=& \lim_{\Delta\to \Z^d} \lim_{n\to\infty} E_{Q_n^{bc}}\left( E(\1_A |\mathcal{F}^\L_\Delta) \1_{\{\text{ all edges in } \EE_\L \text{  are open}\}}\right)\\
&\ge & \lim_{\Delta\to \Z^d} \lim_{n\to\infty} E_{Q_n^{bc}}\left( E(\1_A |\mathcal{F}^\L_\Delta) \1_{\{\text{ all edges in } \EE_\L \text{  are open}\}} \right.\\
& &\left. \1_{\{ \text{an edge between a vertice in } \L \text{ and } \L^c \text{ is open} \}}\right).
\end{eqnarray*}

Given that $ \text{an edge between a vertice in } \L \text{ and } \L^c \text{ is open} $, the weight under  $Q_n^{bc}$ of configurations  $\{\text{ all edges in } \EE_\L \text{  are open}\}$ is positive. Moreover there exists a constant $c>0$ such that this weight is larger than $c$ times the weight of all other allowed configurations. Therefore

\begin{eqnarray*}
& & P (A\cap\{\text{ all edges in } \EE_\L \text{  are open}\})\\
&\ge & c \lim_{\Delta\to \Z^d} \lim_{n\to\infty} E_{Q_n^{bc}}\left( E(\1_A |\mathcal{F}^\L_\Delta) \1_{\{ \text{an edge between a vertice in } \L \text{ and } \L^c \text{ is open} \}}\right)\\
&= & c \lim_{\Delta\to \Z^d}  E_{P}\left( E(\1_A |\mathcal{F}^\L_\Delta) \1_{\{ \text{an edge between a vertice in } \L \text{ and } \L^c \text{ is open} \}}\right)\\
&= & c P(A\cap\{ \text{an edge between a vertice in } \L \text{ and } \L^c \text{ is open} \})>0.
\end{eqnarray*}

\end{proof}

Now we obtain the following theorem which claims that in the stationary case, the number of connected components (necessary unbounded) are equal to zero or one. If it is zero then the fully-connected bond measure is the vacuum state (i.e. charging the null configuration $0^\EE$).

\begin{theo}\label{TheoFCBMconnected}
Let $P$ be in $\LL_s(\lambda,\mu)$ then
$$ P(\Ncc=0 \text{ or } 1)=1.$$
\end{theo}

\begin{proof}

The proof follows the standard Burton and Keane arguments \cite{burkea} for which we give only a sketch of the proof. By shift invariance we show that each ergodic phase of $\tilde P$ (in the extremal decomposition of $P$) has $\tilde P$-almost surely a fix number of connected components in $\N\cup\{\infty\}$. By local modification (Lemma  \ref{localmodif}) this number can not be finite greater than two. We finish the proof by a trifurcation argument showing that this number can not be infinite.

\end{proof} 
\begin{rem} We do not know if Theorem  \ref{TheoFCBMconnected} is valid for $P\in  \LL(\lambda,\mu)$ without assuming that $P$ is stationary. We do not know also if such non-stationary probability measures exist.

\end{rem}

\subsection{DLR equations}

In this section we investigate the DLR equations for the two-parameter model. Let us first define the Gibbs kernels. For any configuration $\tilde \o\in\A$, any bounded $\L\subset \Z^d$ and any $(\lambda,\mu)\in\R^2$ we consider the probability measure on $\{0,1\}^{\EE_\L}$ defined by

$$ Q_\L^{\tilde \o}(\o_\L)=\frac{1}{Z_\L^{\tilde \o}}\1_{\A}(\tilde\o_{\EE\backslash \EE_\L}\o_\L) e^{\lambda N(\omega_\L)}e^{\mu\partial N_\L^{\tilde \o}(\omega_\L)},$$

where $\partial N_\L^{\tilde \o}(\omega_\L)$ is number of closed edges in $\tilde\o_{\EE\backslash \EE_\L}\o_\L$ with at least one of its extremities belonging to an open edge of $\tilde\o_{\EE\backslash \EE_\L}\o_\L$ and moreover this extremity (or the other one) belongs to $\L$. This special form for $\partial N_\L^{\tilde \o}(\omega_\L)$ allows to take account the full dependence of $\o_\L$ in the computation of  $\partial N^{\text{bc}}_{\Delta}(\tilde\o_{\EE_\Delta\backslash \EE_\L}\o_\L)$ with $\Delta$ larger than $\L$. Note that $\partial N_\L^{\tilde \o}$  depends only on $\tilde\o_{L^\oplus}$ with

$$ \L^\oplus=\{i\in \Z^d \text{ such that there exists j in } \L \text{ with } |i-j|\le 2 \}.$$

 Note also the difference between $\partial N_\L^{\tilde \o}(\omega_\L)$ and  $\partial N_\L^{\text{bc}(\tilde \o)}(\omega_\L)$ defined before Definition \ref{defintionQ}. As usual $Z_\L^{\tilde \o}$ is the normalization constant 

$$ Z_\L^{\tilde \o}=\sum_{\o_\L\in\EE_\L}  \1_{\A}(\tilde\o_{\EE\backslash \EE_\L}\o_\L) e^{\lambda N(\omega_\L)}e^{\mu\partial N_\L^{\tilde \o}(\omega_\L)}$$

which is positive since larger than $e^{\lambda N(\tilde \omega_\L)}e^{\mu\partial N_\L^{\tilde \o}(\tilde\omega_\L)}$. 
%
%

%
%

%
%

\begin{defn}\label{definitionDLR} A probability measure  $P$ on $\O$ is a fully-connected bond Gibbs measure with parameter $(\lambda,\mu)\in\R^2$ if $P(\A)=1$ and if for any bounded $\L\subset \Z^d$ and any bounded function $f$ from $\O$ to $\R$

$$\int f(\o)P(d\o)= \int f(\tilde \o_{\EE\backslash \EE_\L}\o_\L) Q_\L^{\tilde \omega}(d\omega_{\L})P(d\tilde \o).$$

We denote by $\G(\lambda,\mu)$ the set of such Gibbs measures and by $\G_s(\lambda,\mu)$ the elements of $\G(\lambda,\mu)$ which are stationary in space.
\end{defn}

To make a connection with the DLR equations \eqref{DLRintro} presented in the introduction, the definition \ref{definitionDLR} above is equivalent to assume that $P(\A)=1$  and

$$ P(.|\tilde \o_{\EE\backslash \EE_\L})= Q_\L^{\tilde \o},$$

for all bounded $\L$ and for $P$-almost all $\tilde \o$.

The next theorem shows how to construct Gibbs measures via accumulation points of $(Q_n^{\text{bc}})$.

\begin{theo}\label{TheoGibbs}
Let $P$ be in $\LL(\lambda,\mu)$ such that  $P(N_{cc}=1)>0$. Then $P(.|N_{cc}=1)$ belongs to $\G(\lambda,\mu)$. 
\end{theo}

From Theorems \ref{TheoFCBMconnected} and \ref{TheoGibbs} we deduce directly the following corollary

\begin{coro}\label{coro1}
Let $P$ be in $\LL_s(\lambda,\mu)$ such that $P(0^\EE)<1$. Then $P(.|\{0^\EE\}^c)$ belongs to  $\G_s(\lambda,\mu)$.
\end{coro}

\begin{proof} (of Theorem \ref{TheoGibbs})

Let $P$ be in $\LL(\lambda,\mu)$ such  $P(N_{cc}=1)>0$. By a standard class monotone argument we have to show that for any bounded set $\L\subset\Z^d$, any local bounded function $f$ from $\O$ to $\R$

\begin{equation}\label{DLRintegral}
\Delta:=|E_P( f\1_{\{N_{cc}=1\}})-E_P(f_\L\1_{\{N_{cc}=1\}})|=0, 
\end{equation}
where $f_\L$ is the function from $\O$ to $\R$ defined by 

\begin{equation}
 f_\L(\tilde\omega)=\int f(\tilde \o_{\EE\backslash \EE_\L}\o_\L) dQ_\L^{\tilde \omega}(\omega_{\L}).
\end{equation} 
 
Without loss of generality we assume that $\Vert f \Vert_{\infty} \le 1$. It ensures in particular that $\Vert f_\L \Vert_{\infty} \le 1$ as well. The main issues to prove DLR equations is that the functions $ f_\L$ and $\1_{\{N_{cc}=1\}}$ are not local. To this end we introduce a collection of events which localizes them. Let us start with the function $\1_{\{N_{cc}=1\}}$. For integers $1\le k'<k<\infty$ we set

$$ L_{k',k}=\left\{
\begin{array}{l}
\o \in \O, \text{ the number of connected components} \\
\text{ in } \o_{\L_k} \text{ intersecting } \L_{k'}  \text{ is equal to 1}
\end{array}
\right\}.
$$

We have the following identity 
\begin{eqnarray}\label{representationlocalncc}
 \{N_{cc}=1\}&=& \bigcup_{k'_0\ge 1}\bigcap_{k'\ge k'_0} \bigcup_{k\ge k'} L_{k',k}. \end{eqnarray}
%

Let us now introduce events to localize $f_\L$. We fix $k_0\ge 1$ such that $\L$ and the support of the function $f$ are included in $\L_{k_0-1}$. For any $k\ge k_0$ we set

$$E_k=\left\{
\begin{array}{l}
\o \in \O, \text{ such that } N(\o_{\EE_k\backslash \EE_\L})\ge 1 \text{ and } \\
\text{ the number of connected components in } \o_{\EE_k\backslash \EE_\L} \text{ intersecting } \L \text{ is equal to}\\
 \text {  the number of connected components in } \o_{\EE\backslash \EE_\L} \text{ intersecting } \L 
\end{array}
\right\}.
$$

In other words, the event $E_k$ ensures that two connected components in $\o_{\EE_k\backslash \EE_\L} $ intersecting $\L$  are not connected using open edges outside $\L_k$. Moreover we assume that there exists at least one such connected component. Note that the event $E_k$ is not local. We define the local version of $Q_\L^{\tilde \o}$ on $E_k$ by

$$ Q_{k,\L}^{\tilde \o}(\o_\L)=\frac{1}{Z^{\tilde \o}_{k,\L}}\1_{\A}(\tilde\o_{\EE_k\backslash \EE_\L}^*\o_\L) e^{\lambda N(\omega_\L)}e^{\mu\partial N_\L^{\tilde \o}(\omega_\L)},$$

where the star at the top of $\tilde \o_{\EE_k\backslash \EE_\L}^*$ means that the connected components of $\tilde \o_{\EE_k\backslash \EE_\L}$ not intersecting $\L$ have been closed. Without this modification isolated bounded connected components could violate the connectivity requirement in $\A$. It is clear that the kernel $\tilde \o \to Q_{k,\L}^{\tilde \o}(.)$ is local since it depends only on $\tilde \o_{\EE_k}$. Moreover for any $\tilde \o\in E_k$ we have that $  Q_{\L}^{\tilde \o}(.)= Q_{k,\L}^{\tilde \o_{\EE_k}}(.)$
   and therefore the function  $\tilde \o\to f_{k,\L}(\tilde\omega)=\int f(\tilde \o_{\EE\backslash \EE_\L}\o_\L) dQ_{k,\L}^{\tilde \omega_{\EE_k}}(\omega_{\L})$ is local and satisfies
   
\begin{equation}\label{localizationID} 
f_{k,\L}\1_{E_k}=f_{\L}\1_{E_k}.
\end{equation}

Let "bc" be a boundary condition such that $P$ is an accumulation point of $( Q_n^{bc})$. For sake a simplicity, we note $P=\lim_{n\to \infty}  Q_n^{bc}$ in omitting the limit under a subsequence. 

We need first to prove that the local events $(L_{k',k})$ and  $(E_k)$ have high probability.

\begin{lemm}\label{lemmetroispoints}
For any $\epsilon>0$ there exists $k\ge k'\ge k_0$ and $n_0\ge 1$ such that for all $n\ge n_0$

\begin{equation}\label{troispoints}
 E_P(|\1_{\{N_{cc}=1\}}-\1_{L_{k',k}}|)\le \epsilon, \quad  P(L_{k',k}\cap E^c_k)\le \epsilon \quad \text{ and }\quad  Q_n^{bc}(L_{k',k}\cap  E^c_k)\le \epsilon.
 \end{equation}
\end{lemm}

\begin{proof}

Let $\epsilon>0$ be a positive real. In identity \eqref{representationlocalncc} the unions are increasing and the intersection decreasing so we fix  $k'(\epsilon)\ge k_0$ and $k_1(\epsilon)\ge k'(\epsilon)$ such that the first point of \eqref{troispoints} holds  for $k'=k'(\epsilon)$ and any $k$ larger than $k_1(\epsilon)$. For the second point in \eqref{troispoints} we note that the sequence of events $(E_k)$ is increasing with 

$$ \bigcup_{k\ge 1} E_k=\{0^\EE\}^c,$$

 therefore $P(\{0^\EE\}^c\cap E^c_k)\to 0$ when $k\to\infty$. Since $L_{k',k}\subset\{0^\EE\}^c$  there exists $k_2(\epsilon)\ge k'(\epsilon)$  such that the second point in \eqref{troispoints}  holds for $k'=k'(\epsilon)$ and $k$ larger than $k_2(\epsilon)$. The last point in \eqref{troispoints} is more delicate because we need an uniform bound with respect to $n$. First we show that there exists $k \ge \max(k_1(\alpha \epsilon/2),k_2(\alpha \epsilon/2))$ and $n_0\ge 1$ such that for all $n\ge n_0$

\begin{equation}\label{truccontra}
 Q_n^{bc}(L_{k'(\alpha \epsilon/2),k}\cap  E^c_k)\le \epsilon,
\end{equation}
   
where $0<\alpha<1$ is a constant, determined later, and which depends only on $\L$, $\lambda$ and $\mu$. Since $\alpha\epsilon/2<\epsilon$, the three inequalities in \eqref{troispoints} hold for $k'=k'(\alpha \epsilon/2)$ and $n_0, k$ obtained for \eqref{truccontra}.

 We make a proof by contradiction in order to show  \eqref{truccontra}. So we assume there exists $\epsilon>0$ such that for any  $k \ge \max(k_0^1(\alpha \epsilon/2),k_0^2(\alpha \epsilon/2))$, there exists an increasing sequence $(n_l)_{l\ge 1}$ of integers such that 

\begin{equation}\label{borneinfapprox}
 Q_n^{bc}(L_{k'(\alpha \epsilon/2),k}\cap  E^c_k)>\epsilon.
 \end{equation}

We denote by $F_k$ the following event

$$F_k=\left\{
\begin{array}{l}
\o \in \O, \text{ there exists at least two connected components in } \o_{\EE_k} \\
\text{ intersecting } \L \text{ and } \L_k^c 
\end{array}
\right\}.
$$
 
 The event $L_{k',k}\cap E^c_k$ ensures that there exists at least two connected components in $\o_{\EE_k\backslash \L}$  intersecting $\L$  and being connected outside $\L_k$. However we do not have that $L_{k',k}\cap E^c_k\subset L_{k',k}\cap F_k$ since both connected components mentioned above can be connected in $\L$. But with a local modification in $\L$ we disconnect them with a uniformly bounded cost from below as in the proof of lemma \ref{localmodif}. So from \eqref{borneinfapprox} and this local modification we deduce that there exists a constant $0<\alpha<1$ (which depends only on $\L$, $\lambda$ and $\mu$)  such that for any  $k \ge \max(k_0^1(\alpha \epsilon/2),k_0^2(\alpha \epsilon/2))$,  

\begin{equation}\label{borneinfapprox2}
  Q_n^{bc}(L_{k'(\alpha \epsilon/2),k}\cap  F_k)>\alpha\epsilon.
 \end{equation}

The event $L_{k'(\alpha \epsilon/2),k}\cap  F_k$ is local and $Q_{n_{l}}^{bc}$ converges to $P$ when $l\to\infty$. We deduce that $P(L_{k'(\alpha \epsilon/2),k}\cap  F_k)\ge \alpha  \epsilon$. By the first point of \eqref{troispoints} we have $E_P(|\1_{\{N_{cc}=1\}}-\1_{L_{k'(\alpha \epsilon/2),k}}|)\le \alpha\epsilon/2$ and therefore

\begin{equation}\label{borneinfapprox3}
 P(\{N_{cc}=1\}\cap F_k)\ge\alpha\epsilon- \alpha\epsilon/2=\alpha\epsilon/2.
 \end{equation}

 Since the sequence of events $(F_k)$ is decreasing with intersection $\{N_{cc}\ge 2\}$, we obtain that 
 
 $$P(\{N_{cc}= 1\}\cap \{N_{cc}\ge 2\})\ge \alpha\epsilon/2>0$$
 
  which is a contradiction. The lemma is proved.

\end{proof}

Let us come back to the proof of \eqref{DLRintegral}. For any $\epsilon>0$ we choose $k\ge k'\ge k_0$ and $n_0\ge 1$ as in Lemma \ref{lemmetroispoints}. Thanks to the localization identity \eqref{localizationID}

\begin{eqnarray*}
\Delta&=&|E_P(f\1_{\{N_{cc}=1\}})-E_P(f_\L\1_{\{N_{cc}=1\}})|\\
&\le &|E_P(f\1_{\{N_{cc}=1\}})-E_P(f_\L\1_{\{L_{k',k}\}})|+\epsilon\\
&\le & |E_P(f\1_{\{N_{cc}=1\}})-E_P(f_{\L}\1_{\{L_{k',k}\}}\1_{E_k})| +2\epsilon \\
&= & |E_P(f\1_{\{N_{cc}=1\}})-E_P( f_{k,\L}\1_{\{L_{k',k}\}}\1_{E_k})| +2\epsilon \\
&\le & |E_P(f\1_{\{N_{cc}=1\}})-E_P( f_{k,\L}\1_{\{L_{k',k}\}})| + 3\epsilon. 
\end{eqnarray*}

Since the function  $f_{k,\L}\1_{L_{k',k}}$ is local we have $\lim_{n\to \infty}  E_{Q_n^{bc}}(f_{k,\L} )=E_P( f_{k,\L})$. so for $n$ large enough (larger than $n_0$)

\begin{eqnarray*}
\Delta & \le & |E_P(f\1_{\{N_{cc}=1\}})-  E_{Q_n^{bc}}(f_{k,\L} \1_{\{L_{k',k}\}})|+4\epsilon\\
& \le & |E_P(f\1_{\{N_{cc}=1\}})-  E_{ Q_n^{bc}}(f_{k,\L}  \1_{\{L_{k',k}\}}\1_{E_k})|+5\epsilon.
\end{eqnarray*}
For any $\tilde \omega\in\A$
 \begin{eqnarray*}
 (f_\L\1_{\{L_{k',k}\}} \1_{E_k})(\tilde \o)&=& \int f(\tilde \o_{\EE\backslash \EE_\L}\o_\L) dQ_\L^{\tilde \omega}(\omega_{\L})\1_{\{L_{k',k}\}}(\tilde \o) \1_{E_k}(\tilde \o)\\
 & =& \frac{1}{Z_\L^{\tilde \o}} \sum_{\o_\L\in\O_\L} \1_{\{L_{k',k}\}}(\tilde \o) \1_{E_k}(\tilde \o)\1_{\A}(\tilde\o_{\EE\backslash \EE_\L}\o_\L)\\
 & & \qquad f(\tilde \o_{\EE\backslash \EE_\L}\o_\L) e^{\lambda N(\omega_\L)}e^{\mu\partial N_\L^{\tilde \o}(\omega_\L)}\\
  & =& \frac{1}{Z_\L^{\tilde \o}} \sum_{\o_\L\in\O_\L} \1_{\{L_{k',k}\}}(\tilde\o_{\EE\backslash \EE_\L}\o_\L) \1_{E_k}(\tilde\o_{\EE\backslash \EE_\L}\o_\L)\1_{\A}(\tilde\o_{\EE\backslash \EE_\L}\o_\L)\\
 & & \qquad f(\tilde \o_{\EE\backslash \EE_\L}\o_\L) e^{\lambda N(\omega_\L)}e^{\mu\partial N_\L^{\tilde \o}(\omega_\L)}\\
 &=& (f\1_{\{L_{k',k}\}} \1_{E_k})_\L(\tilde \o).
\end{eqnarray*}

 Now from a simple finite volume DLR equation for $Q_n^{bc}$, we can substitute $f$ to $f_\L$ and obtain for $n$ large enough

\begin{eqnarray*}
\Delta & \le &  \Bigg|E_P(f\1_{\{N_{cc}=1\}})-  E_{ Q_n^{bc}}(f \1_{\{L_{k',k}\}}\1_{E_k}) \Bigg|+6\epsilon\\
& \le & \Bigg|E_P(f\1_{\{N_{cc}=1\}})-  E_{Q_n^{bc}}(f\1_{\{L_{k',k}\}}) \Bigg|+7\epsilon\\
& \le & \Bigg|E_P(f\1_{\{N_{cc}=1\}})-  E_{P}(f\1_{\{L_{k',k}\}}) \Bigg|+8\epsilon\\
& \le & \Bigg|E_P(f\1_{\{N_{cc}=1\}})-  E_{P}(f\1_{\{N_{cc}=1\}}) \Bigg|+9\epsilon\\
& =& 9 \epsilon.
\end{eqnarray*}

This inequality holds for any $\epsilon>0$. Therefore $\Delta=0$ and the theorem is proved.

\end{proof}
%
%
%
%
%
%
%
%
%
%
%
%
%

\subsection{Pressure}

In this section we study the pressure of the  model. Let us recall that $Z_{n}^{\text{bc}}(\lambda,\mu)$ is the partition function of  $Q_n^{\text{bc}}$;

$$ Z_{n}^{\text{bc}}(\lambda,\mu):=\sum_{\o_{\L_n}\in\O_{\L_n}} \1_{\left\{\Ncc^{\text{bc}}(\omega_{\L_n})=1\right\}} e^{\lambda N_\L(\o_{\L_n})} e^{\mu \partial N^{\text{bc}}_{\L_n}(\o_{\L_n})}.$$

\begin{prop}\label{propexistencepression}

The following limit exists in $[0,+\infty)$ and is called pressure with wired boundary condition

$$\PP(\lambda,\mu)=\lim_{n\to \infty} \frac{ \log(Z_{n}^{\text{wired}}(\lambda,\mu))}{\# \EE_n}.$$

Moreover for any compact set $\mathcal K\subset \R^2$ there exists a constant $C>0$ such that for  any $\lambda,\mu\in\mathcal K$ and any  boundary condition "bc" 

\begin{equation}\label{comparisonpressure1}
\left | \frac{ \log(Z_{n}^{\text{bc}}(\lambda,\mu))}{\# \EE_n} -  \frac{ \log(Z_{n}^{\text{wired}}(\lambda,\mu))}{\# \EE_n}\right| \le \frac{C}{\sqrt{n}}
\end{equation}

and for any $\tilde \o\in\A$ such that at least one vertex in $\L_n$ belongs to an open edge of $\tilde \o$

\begin{equation}\label{comparisonpressure2}
\left | \frac{ \log(Z_{n}^{\tilde \o}(\lambda,\mu))}{\# \EE_n} -  \frac{ \log(Z_{n}^{\text{wired}}(\lambda,\mu))}{\# \EE_n}\right| \le \frac{C}{\sqrt{n}}.
\end{equation}

In particular, the pressures with boundary condition "bc" or $\tilde \o\in\A$ exist and are equal to the pressure with wired boundary condition.

\end{prop}

\begin{proof}

We first prove that the pressure  with wired boundary condition exists by following a standard bloc decomposition. For any $2\le m\le n$, we consider the Euclidean division n=km+l with $0\le l< m$ and  $k\ge 0$. Let $(\L_{m-1}^i)_{1\le i \le k^d}$ be a family of $k^d$ disjoint sets inside $\L_n$ where each $\L_{m-1}^i$ is a translation of $\L_{m-1}$.  We denote by $\EE^{\text{out}}_{n,m}$ the edges in $\EE_n$ which are not inside the boxes $(\L_{m-2}^i)_{1\le i \le k^d}$;
$$\EE^{\text{out}}_{n,m}:=\EE_n\backslash \left(\bigcup_{1\le i \le k^d} \EE_{\L_{m-2}^i}\right).$$

 A bloc decomposition of the partition function implies 

\begin{eqnarray*}
Z_{n}^{\text{wired}}(\lambda,\mu) &=& \sum_{\o_{\L_n}\in\O_{\L_n}} \1_{\left\{\Ncc^{\text{wired}}(\omega_{\L_n})=1\right\}} e^{\lambda N_\L(\o_{\L_n})} e^{\mu \partial N^{\text{wired}}_{\L_n}(\o_{\L_n})}\\
&\ge & \sum_{
\begin{subarray}{c}
\o_{\L_n}\in\O_{\L_n}\\
 \text{the edges } \o_{\EE^{\text{out}}_{n,m}} \text{ are open.}
\end{subarray} 
}
  \1_{\left\{\Ncc^{\text{wired}}(\omega_{\L_n})=1\right\}} e^{\lambda N_\L(\o_{\L_n})} e^{\mu \partial N^{\text{wired}}_{\L_n}(\o_{\L_n})}\\
 & \ge & e^{\lambda \# \EE^{\text{out}}_{n,m}} e^{2dk^dm^{d-1}\min(\mu,0)}\prod_{1\le i \le k^d}  Z_{\L_{m-2}^i}^{\text{wired}}(\lambda,\mu),
\end{eqnarray*}

where the combinatorial term $2dk^dm^{d-1}$ in the exponential is the maximal number of edges in $\cup_{1\le i \le k^d} \EE_{\L_{m-1}^i}$ sharing a vertex with an edge in $\EE^{\text{out}}_{n,m}$. We denote by $N_m$ the following limit $\#\EE_n/k^d$ when $n$ goes to infinity. It is easy to see that $N_m$ exists and that it is equivalent to $\#\EE_{m}$ when $m\to \infty$. Therefore

$$ \liminf_{n\to \infty} \frac{1}{\#\EE_n} \ln(Z_{n}^{\text{wired}}(\lambda,\mu)) \ge  \frac{1}{N_{m}} \left(\ln(Z_{m-2}^{\text{wired}}(\lambda,\mu))- Cm^{d-1}\right),$$

where $C\ge 0$ is a constant (depending only on $d$) taking account all boundary terms. This inequality holds for each $m\ge 2$. So, letting $m$ tends to infinity

$$ \liminf_{n\to \infty} \frac{1}{\#\EE_n} \ln(Z_{n}^{\text{wired}}(\lambda,\mu)) \ge  \limsup_{m\to \infty} \frac{1}{N_m} \ln(Z_{m-2}^{\text{wired}}(\lambda,\mu))= \limsup_{m\to \infty} \frac{1}{\#\EE_m} \ln(Z_{m}^{\text{wired}}(\lambda,\mu)),  $$

which proves that the limit exists in $\R\cup\{\pm \infty\}$. Simple combinatorial arguments show that this limit is larger than $0$ and smaller than $\log(2)+\max(0,\lambda)+\max(0,\mu)$, which excludes the case where $\PP(\lambda,\mu)$ is plus or minus infinity.  

Let us now prove \eqref{comparisonpressure1}. The proof of \eqref{comparisonpressure2} is similar excepted a detail which we provide at the end. For any $n\ge 2$ we consider the following set $\EE_n^g$ of edges in $\EE_n$ (the symbol $g$ is related to the grid built by the edges considered in this set). An edge $e=(k,k')\in\EE_n$ belongs to  $\EE_n^g$ if it is at the boundary (i.e. $e\notin \EE_{n-1}$) or if at least one of the $d$ coordinates $i$ of $k$ (or $k'$)  satisfies $i=[[\sqrt{n}[i/\sqrt{n}]] $  where $[.]$ denotes the integer part of a real number. 

In other words the set of edges $\EE_n^g$ fills the boundary of $\L_n$ and separate $\L_n$ with large hyperplanes (in all directions) with inter-distance around $\sqrt{n}$. The complement $\EE_n\backslash \EE_n^g$ is composed with approximatively $(2n)^{d/2}$ disjoints cubes of edges with length side $\sqrt{n}$. The number of edges in $\EE_n^g$ is of order $n^{d-1/2}$. Let $\mathcal K$ be a compact set in $\R^2$ and let $\lambda,\mu\in\mathcal K$. For every boundary condition bc and bc'
 
 \begin{eqnarray*}
 Z_{n}^{\text{bc}}(\lambda,\mu) &=& \sum_{\o_{\L_n}\in\O_{\L_n}} \1_{\left\{\Ncc^{\text{bc}}(\omega_{\L_n})=1\right\}} e^{\lambda N_\L(\o_{\L_n})} e^{\mu \partial N^{\text{bc}}_{\L_n}(\o_{\L_n})}\\
  &=& \sum_{\o_{\L_n}\in\O_{\L_n}} \1_{\{\Ncc^{\text{bc}}(\omega_{\L_n})=1\text{ and at least one edge of } \EE_n^g \text{ is open}\}} e^{\lambda N_\L(\o_{\L_n})} e^{\mu \partial N^{\text{bc}}_{\L_n}(\o_{\L_n})}\\
  & &+  \sum_{\o_{\L_n}\in\O_{\L_n}} \1_{\{\Ncc^{\text{bc}}(\omega_{\L_n})=1\text{ and all edges of } \EE_n^g \text{ are closed}\}} e^{\lambda N_\L(\o_{\L_n})} e^{\mu \partial N^{\text{bc}}_{\L_n}(\o_{\L_n})}.\\
\end{eqnarray*}
  
 If all edges of $\EE_n^g$ are closed it means that the unique open cluster is included in one cube of the grid. The number of cubes is of order $n^{d/2}$, the number of edges in each cube if of order $n^{d/2}$ and so the number of  configurations in each cube is of order $2^{(n^{d/2})}$. Moreover the weight $e^{\lambda N_\L(\o_{\L_n})} e^{\mu \partial N^{\text{bc}}_{\L_n}(\o_{\L_n})}$ is of order $max(1,e^\lambda,e^\mu)^{(n^{d/2})}$. Therefore there exists a constant $C_1$ such that
 
 $$  \sum_{\o_{\L_n}\in\O_{\L_n}} \1_{\{\Ncc^{\text{bc}}(\omega_{\L_n})=1\text{ and all edges of } \EE_n^g \text{ are closed}\}} e^{\lambda N_\L(\o_{\L_n})} e^{\mu \partial N^{\text{bc}}_{\L_n}(\o_{\L_n})} \le C_1^{(n^{d/2})}.$$
  
For the other sum, we compare it with the same sum but for which all edges in $\EE_n^g$ are open;  
 
 $$  \sum_{\o_{\L_n}\in\O_{\L_n}} \1_{\{\Ncc^{\text{bc}}(\omega_{\L_n})=1\text{ and all edges of } \EE_n^g \text{ are open}\}} e^{\lambda N_\L(\o_{\L_n})} e^{\mu \partial N^{\text{bc}}_{\L_n}(\o_{\L_n})}.$$ 
   
For each configuration for which "at least one edge in $\EE_n^g$ open" we associate easily a configuration for which all edges are open by opening the edges in $\EE_n^g$ which are closed. The modification of the weight $e^{\lambda N_\L(\o_{\L_n})} e^{\mu \partial N^{\text{bc}}_{\L_n}(\o_{\L_n})}$ is of order a constant to the power of the cardinal of $\EE_n^g$. That provides a multiplicative factor $C^{(n^{d-1/2})}$ for some constant $C$. Moreover the association mentioned above is not a bijection since two different configurations with "at least one edge in $\EE_n^g$ open" can produce the same configuration after opening all edges in $\EE_n^g$. However any configuration with "all edges in $\EE_n^g$ are open" comes from at most $2^{\#\EE_n^g}$ different configurations with   "at least one edge in $\EE_n^g$ open". We deduce that there exists a constant $C_2$ such 

\begin{eqnarray*}
& &  \sum_{\o_{\L_n}\in\O_{\L_n}} \1_{\{\Ncc^{\text{bc}}(\omega_{\L_n})=1\text{ and at least one edge of } \EE_n^g \text{ is open}\}} e^{\lambda N_\L(\o_{\L_n})} e^{\mu \partial N^{\text{bc}}_{\L_n}(\o_{\L_n})}\\
&  \le & C_2^{(n^{d-1/2})} \sum_{\o_{\L_n}\in\O_{\L_n}} \1_{\{\Ncc^{\text{bc}}(\omega_{\L_n})=1\text{ and all edges of } \EE_n^g \text{ are open}\}} e^{\lambda N_\L(\o_{\L_n})} e^{\mu \partial N^{\text{bc}}_{\L_n}(\o_{\L_n})}.\\
\end{eqnarray*}
 
Now noting that if all edges of $\EE_n^g $ are open then $\{ \Ncc^{\text{bc}}(\omega_\L)=1\}=\{ \Ncc^{\text{bc'}}(\omega_\L)=1\}$ and $\partial N^{\text{bc}}_{\L_n}(\o_{\L_n})=\partial N^{\text{bc'}}_{\L_n}(\o_{\L_n})$ we obtain 
  
  \begin{eqnarray}\label{majorationZ}
 Z_{n}^{\text{bc}}(\lambda,\mu) 
 & \le &  C_2^{n^{d-1/2}} \sum_{\o_{\L_n}\in\O_{\L_n}} \1_{\{\Ncc^{\text{bc'}}(\omega_{\L_n})=1\text{ and all edges of } \EE_n^g \text{ are open}\}} e^{\lambda N_\L(\o_{\L_n})} e^{\mu \partial N^{\text{bc'}}_{\L_n}(\o_{\L_n})}\nonumber \\
  & & + C_1^{n^{d/2}}  \nonumber\\
  & \le &  C_2^{n^{d-1/2}} \sum_{\o_{\L_n}\in\O_{\L_n}} \1_{\{\Ncc^{\text{bc'}}(\omega_{\L_n})=1\}} e^{\lambda N_\L(\o_{\L_n})} e^{\mu \partial N^{\text{bc'}}_{\L_n}(\o_{\L_n})}\nonumber\\
  & & + C_1^{n^{d/2}} \nonumber \\
  & = &  C_2^{n^{d-1/2}}  Z_{n}^{\text{bc'}}(\lambda,\mu)  + C_1^{n^{d/2}}.
 \end{eqnarray}

Since a configuration with all edges at the boundary open and all edges in the bulk closed is allowed we deduce that $Z_{n}^{\text{bc'}}\ge C_3^{(n^{d-1})}$ for some constant $C_3>0$. Therefore $C_1^{n^{d/2}} \le C_1^{n^{d/2}}(1/C_3)^{(n^{d-1})}Z_{n}^{\text{bc'}}$. At the end there exists a constant $C>0$ such that 

\begin{equation}\label{inegalitepression}
Z_{n}^{\text{bc}}(\lambda,\mu)\le C^{(n^{d-1/2})} Z_{n}^{\text{bc'}}(\lambda,\mu).
\end{equation}

Passing to the logarithm and dividing by $\#\EE_n$ we find that

\begin{equation}
 \frac{ \log(Z_{n}^{\text{bc}}(\lambda,\mu))}{\# \EE_n} -  \frac{ \log(Z_{n}^{\text{bc'}}(\lambda,\mu))}{\# \EE_n} \le \log(C)\frac{n^{d-1/2}}{\#\EE_n}.
\end{equation}

Noting that $\# \EE_n$ is of order $n^d$ and applying twice the previous inequality with one time bc="wired" and a second time with bc'="wired" the inequality \eqref{comparisonpressure2} follows.

Involving the proof of \eqref{comparisonpressure2}. The scheme of the proof is exactly the same excepted we need the assumption that  at least one vertex in $\L_n$ belongs to an edge of $\tilde \o$. Without these assumption the first inequality in \eqref{majorationZ} is wrong.
Indeed if the connected component of $\tilde \o$ does not hit $\L_n$ then $\Ncc^{\tilde \o}(\omega_{\L_n})=1$ if and only if $\omega_{\L_n}=0^{\EE_n}$. Therefore the assertion $\{ \Ncc^{\text{bc}}(\omega_{\L_n})=1\}=\{ \Ncc^{\tilde \o}(\omega_{\L_n})=1\}$ if all edges of $\EE_n^g $ are open is not true in general. This equivalence is crucial to prove \eqref{majorationZ}. The rest of the proof is the same.

\end{proof}
%
%
%
%
%
%
%
%
%
%
%
%
%

Let us finish this section with standard properties of the pressure function $\PP$.

\begin{prop}\label{propconvexitypressure}

The function ($\lambda,\mu)\to \PP(\lambda,\mu)$ is  convex and non-decreasing with respect to each variable $\lambda$ or $\mu$.
\end{prop}

\begin{proof}

Let us recall that 

$$ Z_{n}^{\text{wired}}(\lambda,\mu) = \sum_{\o_{\L_n}\in\O_{\L_n}} \1_{\left\{\Ncc^{\text{wired}}(\omega_{\L_n})=1\right\}} e^{\lambda N_{\L_n}(\o_{\L_n})} e^{\mu \partial N^{\text{wired}}_{\L_n}(\o_{\L_n})}.$$

So by simple calculus we have

\begin{equation}\label{increasespressure}
 \frac{\partial \log(Z_n^{\text{wired}})}{\partial \lambda} =E_{Q_{\L_n}^{\text{wired}}}(N_{\L_n}),\quad
 \frac{\partial \log(Z_n^{\text{wired}})}{\partial \mu} = E_{Q_{\L_n}^{\text{wired}}}(\partial N_{\L_n}^{\text{wired}}),
 \end{equation}

 $$ \frac{\partial^2 \log(Z_n^{\text{wired}})}{\partial^2 \lambda} =  \text{Var}_{Q_{\L_n}^{\text{wired}}}(N_{\L_n}),\quad \frac{\partial^2 \log(Z_n^{\text{wired}})}{\partial^2 \mu} =  \text{Var}_{Q_{\L_n}^{\text{wired}}}(\partial N_{\L_n}^{\text{wired}})  $$

and 

 $$ \frac{\partial^2 \log(Z_n^{\text{wired}})}{\partial \lambda\partial \mu} =  \text{Cov}_{Q_{\L_n}^{\text{wired}}}(N_{\L_n},\partial N_{\L_n}^{\text{wired}}),$$

which ensures that the function $(\lambda,\mu)\mapsto  \log(Z_{\L_n}^{\text{wired}}(\lambda,\mu))$ is non-decreasing (with respect to each variable $\lambda$ or $\mu$) and convex. Dividing by $\#\EE_n$ and passing to the limit, the function $(\lambda,\mu)\mapsto  \PP(\lambda,\mu))$ is non-decreasing and convex as well.

\end{proof}

\subsection{Existence of the threshold}\label{Section_transition}

In this section we show that for any $\mu\in\R$, there exists a threshold $\lambda^*(\mu)\in\R$ such that  $\G_s(\lambda,\mu)$ is empty for $\lambda<\lambda^*(\mu)$ and not empty for $\lambda>\lambda^*(\mu)$. We do not know in general what happens at criticality $\lambda=\lambda^*(\mu)$. Recall that there is no stochastic monotony between elements in $\G(\lambda,\mu)$ when $\lambda$ is increasing and therefore the existence of a threshold $\lambda^*(\mu)$ is not obvious. We define $\lambda^*(\mu)$ as follows

\begin{equation}\label{definitionpcd} 
\lambda^*(\mu)=\sup\Big\{\lambda\in\R, \PP(\lambda,\mu)= 0\Big\},
\end{equation}
with the convention $\lambda^*(\mu)=-\infty$ if the set is empty. Actually in Section \ref{Section_bounds} below we show that $\lambda^*(\mu)$ is finite for any $\mu\in\R$. By proposition  \ref{propexistencepression} and  \ref{propconvexitypressure} the function $\PP$ is non-negative and convex. We deduce that the set $\{(\lambda,\mu)\in\R^2, \PP(\lambda,\mu)=0\}$ is convex and so the function $\mu\to \lambda^*(\mu)$ is concave (in particular continuous).

\begin{theo}\label{TheoremThreshold} In any dimension $d\ge 2$ and for all $\lambda, \mu\in\R$ 
\begin{itemize}
\item if  $\lambda>\lambda^*(\mu)$ then  any $P\in\LL_s^{\text{per}}(\lambda,\mu)$ is not equal to $\delta_{0^\EE}$. In particular $\G_s(\lambda,\mu)\neq\emptyset$. 
\item if  $\lambda<\lambda^*(\mu)$ then 
 $\LL_s(\lambda,\mu)=\{\delta_{0^\EE}\}$ and $\G_s(\lambda,\mu)=\emptyset$.
\end{itemize}

\end{theo}

\begin{proof} 

The proof of the theorem is based on these two main assertions: 

\begin{itemize}[label=-]

\item{\bf [Assertion 1]}:  If  $\lambda>\lambda^*(\mu)$ then any $P\in\LL_s^{\text{bc}}(\lambda,\mu)$ is not equal to $\delta_{0^\EE}$
 
\item {\bf [Assertion 2]}: For any $\mu\in\R$ and any $\lambda<\lambda^*(\mu)$,

$$ \lim_{n\to \infty} \frac{1}{\#\EE_n} \sup_{
\begin{subarray}{c}
\tilde \omega \in \A 
 \end{subarray}
 } \int N dQ_{\L_n}^{\tilde \omega}=0.$$

\end{itemize}
 
Assertion 1 and corollary \ref{coro1} show the first item of the theorem. For the second item let us start to show that if $\lambda <\lambda^*(\mu)$ then $\G_s(\lambda,\mu)=\emptyset$. We make a proof by contradiction in assuming that there exists $P\in\G_s(\lambda,\mu)$. By Assertion 2, for any fix edge $e$ 

\begin{eqnarray*}
P(\text{ the edge } e \text{ is open})&=& \lim_{n\to \infty}\frac{1}{\#\EE_n}   \int N_{\L_n}(\omega_{\EE_n}) P(d\omega)\\
&=&  \lim_{n\to \infty}\frac{1}{\#\EE_n}  \int \int N_{\L_n} dQ_{\L_n}^{\tilde \omega} P(d\tilde \omega)\\
&\le & \lim_{n\to \infty}  \frac{1}{\#\EE_n} \sup_{
\begin{subarray}{c}
\tilde \omega \in \A 
 \end{subarray}
 } \int N dQ_{\L_n}^{\tilde \omega}\\
&=&0.
\end{eqnarray*}

This implies that $P=\delta_{0^\EE}$. It is in contradiction with $P(\A)=1$.  It remains to prove that  $\LL_s(\lambda,\mu)=\{\delta_{0^\EE}\}$. If it is not the case there exists $P\in\LL_s(\lambda,\mu)$ such that $P\neq \delta_{0^\EE}$. But by Corollary \ref{coro1} the set $\G_s(\lambda,\mu)$ should be not empty which it is not the case. The theorem is proved.

\end{proof}

 \begin{proof} (of Assertion 1)
 
 Let $\lambda>\lambda^*(\mu)$ and so by definition  $\PP(\lambda,\mu)>0$. We have to that any accumulation point of  $(Q_n^{\text{per}})_{n\ge 1}$ (for any "bc") is not reduced to the null configuration $0^\EE$. For simplicity we write $ Q_n^{\text{per}}\to Q^{\text{per}}$ although it is only for a subsequence. We define the specific entropy of any stationary probability measure $P$ on $\O$ with respect to $\P_{1/2}$ by
$$I(P)=\lim_{n\to\infty} \frac{1}{\#\EE_n} I(P_{\L_n}| \P^{\L_n}_{1/2}),$$   
where $P_{\L_n}$ is the restriction of $P$ on $\O_n$ and $I(P_{\L_n}| \P^{\L_n}_{1/2})$ is simply the standard relative entropy of  $P_{\L_n}$ with respect to $\P^{\L_n}_{1/2}$. Following chapater 15 in \cite{georgii} $I(P)$ is well defined and 

\begin{equation} \label{borneHO}
I(Q^{\text{per}}) \le \limsup_{n\to \infty} \frac{1}{\#\EE_n} I(Q_n^{\text{per}}| \P^{\L_n}_{1/2}).
\end{equation}
 
Recall the expression of  $Q_n^{\text{per}}$

\begin{eqnarray*}
Q_{n}^{\text{per}}(\o_{\L_n})&=&\frac{1}{Z_n^{\text{per}}} \1_{\left\{\Ncc^{\text{per}}(\omega_{\L_n})=1\right\}} e^{\lambda N_{\L_n}(\o_{\L_n})} e^{\mu \partial N^\text{per}_{\L_n}(\o_{\L_n})}\\
&=& \frac{2^{\#\EE_n}}{Z_n^{\text{per}}} \1_{\left\{\Ncc^{\text{per}}(\omega_{\L_n})=1\right\}} e^{\lambda N_{\L_n}(\o_{\L_n})} e^{\mu \partial N^\text{per}_{\L_n}(\o_{\L_n})}\P^{\L_n}_{1/2}(\o_{\L_n}).
\end{eqnarray*}

By Proposition \ref{propexistencepression} we have
\begin{eqnarray*}
I(Q^{\text{per}})& \le  & \limsup_{n\to \infty}   \frac{1}{\#\EE_n} I\left ( Q_n^{\text{per}}| \P^{\L_n}_{1/2}\right)\\
& = & \log(2)+ \limsup_{n\to \infty} -\frac{\log(Z_n^{\text{per}})}{\#\EE_n} +   \frac{\lambda}{\#\EE_n}  \int N_{\L_n} dQ_n^{\text{per}} +   \frac{\mu}{\#\EE_n}  \int \partial N^{\text{per}}_{\L_n} dQ_n^{\text{per}}\\
& < & \log(2)+ \lambda \limsup_{n\to \infty}  \bar Q_n^{\text{per}}(\text{the edge } e \text{ is open})+ \max(\mu,0)\times\\
& &  \times\limsup_{n\to \infty}Q_n^{\text{per}}(\text{the edge } e \text{ is closed and one of edges } (e_i)_{1\le i\le 2(d+1)} \text{ is open} )
\end{eqnarray*} 

where $e$ is any fixed edge in $\EE$ and $(e_i)_{1\le i\le 2(d+1)}$ the $2(2d-1)$ neighbour edges of the edge $e$. Since $ Q_n^{\text{per}}\to Q^{\text{per}}$

\begin{eqnarray*}
I_{1/2}(Q^{\text{per}})
& <& \log(2)+ \lambda \limsup_{n\to \infty}   Q_n^{\text{per}}(\text{the edge } e \text{ is open})\\
& & + \sum_{i=1}^{2(d+1)} \max(\mu,0)  \limsup_{n\to \infty}  Q_n^{\text{per}}(\text{the edge } e_i \text{ is open})\\ 
& = & \log(2)+  \big(\lambda+2(2d-1)\max(\mu,0)\big) Q^{\text{per}}(\text{the edge } e \text{ is open}).
\end{eqnarray*} 

 If $  Q^{\text{per}}(\text{the edge } e \text{ is open})>0$ then we have exactly what we want (i.e. $ Q^{\text{per}}$ is not $\delta_{0^\EE}$). If $  Q^{\text{per}}(\text{the edge } e \text{ is open})=0$  then $I_{1/2}(Q^{\text{per}})<\log(2)=I_{1/2}(\delta_{0^\EE})$, which is enough to claim that $ Q^{\text{bc}}$ is not $\delta_0^\EE$. In any case $ Q^{\text{per}}$ is not the probability measure $\delta_{0^\EE}$. 

\end{proof}

 \begin{proof} (of Assertion 2)

Let $\mu\in\R$, $\lambda<\lambda^*(\mu)$
and  $\tilde \omega\in\A$. There is two cases to distinguish:

{\it -First case:}  no vertex in $\L_n$ belongs to an edge of $\tilde \o$. In this case the distribution $Q_{\L_n}^{\tilde \omega}$ is reduced to the probability measure $\delta_{0^{\EE_n}}$ and therefore

$$ \int N dQ_{\L_n}^{\tilde \omega}=0.$$

{\it -Second case:}  at least one vertex in $\L_n$ belongs to an edge of $\tilde \o$. In this case we use estimates from Proposition \ref{propexistencepression} to prove the assertion. Precisely, for any $n\ge 1$  we have

\begin{equation}\label{derive}
\frac{1}{\#\EE_n} \int N dQ_{\L_n}^{\tilde \omega}=  \frac{1}{\#\EE_n} \frac{\partial \log(Z_n^{\tilde \omega})}{\partial \lambda}.
 \end{equation}

Since that the function $\lambda\mapsto  \log(Z_{\L_n}^{\tilde \omega})$ is increasing and convex we deduce that for any $\lambda< \lambda' <\lambda^*(\mu)$

\begin{eqnarray*}
 \frac{1}{\#\EE_n}  \int N dQ_{\L_n}^{\tilde \omega} & \le & \frac{1}{\lambda'-\lambda} \left[ \frac{1}{\#\EE_n}\Big(\log(Z_n^{\tilde\o}(\lambda',\mu)-\log(Z_n^{\tilde\o}(\lambda,\mu)\right]\\
 & \le & \frac{1}{\lambda'-\lambda} \left[ \frac{1}{\#\EE_n}\Big(\log(Z_n^{\text{wired}}(\lambda',\mu)-\log(Z_n^{\text{wired}}(\lambda,\mu)\Big)+2C/\sqrt{n}\right],
 \end{eqnarray*}

where $C>0$ comes from \eqref{comparisonpressure2}. The right term  does not depend on $\tilde \omega$ and tends to $0$ when $n$ tends to infinity. Assertion 2 is proved.
\end{proof}

\subsection{Explicit value for the threshold}

Based on the representation of the infinite cluster for a Bernoulli bond percolation model (see Proposotion \ref{propbernoulli}), we identify explicitly the value of  $\lambda^*(\mu)$ for  $\mu$ small enough. Recall that  $p_c(d)$ is the Bernoulli percolation threshold defined in \eqref{thresholdperco}.

\begin{theo}\label{theoremexplicit}
For any $\mu\le \log(1-p_c(d))$, we have $\lambda^*(\mu)=\log(1-e^\mu)$. Moreover for $\log(1-p_c(d))<\mu<0$ then $\lambda^*(\mu)>\log(1-e^\mu)$.

\end{theo}

\begin{proof}

Note first that for any $0<p<1$ and $\lambda=\log(p)$ and $\mu=\log(1-p)$ the weight $ \1_{\left\{\Ncc^{\text{wired} }(\omega_\L)=1\right\}} e^{\lambda N_\L(\o_\L)} e^{\mu \partial N^{\text{wired}}(\o_\L)}$ in the Definition \ref{defQn} is simply the probability, under $\P_p^\L$, that the open edges in $\o_\L$ is the collections of clusters hitting the boundary $\partial \L$. Therefore the partition function $Z_\L^{\text{wired}}(\lambda,\mu)$ is the sum of these probabilities and is equal to one.  Therefore the pressure $\PP(\lambda,\mu)=0$ and by definition \eqref{definitionpcd} $\lambda^*(\mu)\ge \lambda$. We deduce that for all  $\mu<0$ 

\begin{equation}\label{ineq1}
 \lambda^*(\mu)\ge  \log(1-e^\mu).
\end{equation}

It remains to prove that this inequality is  an equality if and only if  $\mu\le \log(1-p_c(d))$.

 Let us start by proving that for  $\mu<\log(1-p_c(d))$ it is an equality. We fix $\mu< \log(1-p_c(d))$ and $\lambda=\log(1-e^\mu)$. By Proposition \ref{propbernoulli} and Corollary \ref{coro1}, the set $\G_s(\lambda,\mu)$ contains the probability measure $\P_p^\infty$ (i.e. the distribution of the infinite cluster for the Bernoulli bond percolation) for $p=1-e^\mu>p_c(d)$. It implies that $\G_s(\lambda,\mu)\neq \{\delta_{0^\EE}\}$ and by Theorem \ref{TheoremThreshold} we deduce that

\begin{equation}\label{ineq2}
 \lambda^*(\mu)\le  \log(1-e^\mu).
\end{equation}

Let us now prove that for $\mu>\log(1-p_c(d))$ the strict inequality $\lambda^*(\mu)>  \log(1-e^\mu)$ holds. We have to show that for $\lambda>\log(1-e^\mu)$ small enough the pressure $\PP(\lambda,\mu)=0$. Let $\alpha>0$, we have

\begin{eqnarray*}
Z_{n}^{\text{wired}}(\log(1-e^\mu)+\alpha,\mu) &=&\sum_{\o_{\L_n}\in\O_{\L_n}} \1_{\left\{\Ncc^{\text{wired}}(\omega_\L)=1\right\}}e^{\alpha N_\L(\o_{\L_n})} e^{\log(1-e^\mu) N_\L(\o_{\L_n})} e^{\mu \partial N^{\text{wired}}_{\L_n}(\o_{\L_n})}\\
&=& E_{\P_p}\left(e^{\alpha N_n^{\text{boundary}}}\right),
\end{eqnarray*}

where, as usual $p=1-e^\mu$, and $N_n^{\text{boundary}}$ is the random variates counting the number of edges inside open clusters in $\L_n$ hitting the boundary $\partial \L_n$. For every $\Delta\subset \Z^d$ we denote by $N^{\Delta}$  the variable counting the number of edges in the clusters hitting a vertex $i\in\Delta$ (we do not assume that the clusters are included in $\L_n$). So $N_n^{\text{boundary}}$ is dominated by $N^{\partial \L_n}$. By a standard coupling argument, for any distinct $i,j\in\Z^d$, the distribution of $N^{\{i,j\}}$ under $\P_p$ is stochastically dominated by the convolution of distributions of $N^{i}$ and $N^{j}$ (both under $\P_p$). Indeed, given $N^{i}$, the distribution of $N^{\{i,j\}}$ is the sum of $N^{i}$ and the number of  edges in the clusters hitting $\{i,j\}$ not yet explored from $j$. This random number is clearly stochastically dominated by the distribution of $N^{j}$ itself. If we repeat this coupling for all vertices in the boundary, we obtain that the distribution of $N_n^{\text{boundary}}$ is stochastically dominated by the convolution of distributions $(N^{i})_{i\in\partial \L_n}$. We deduce that

\begin{eqnarray*}
Z_{n}^{\text{wired}}(\log(1-e^\mu)+\alpha,\mu) &\le & \prod_{i\in\partial\L_n} E_{\P_p}\left(e^{\alpha N^{i}}\right).
\end{eqnarray*}

Note that random variables $(N^{i})$ are identical distributed under $\P_p$. Moreover $p=1-e^\mu<p_c(d)$  and so the subcritical regime occurs. Therefore the variable $N^{0}$ under $\P_p$, equals to the size of the cluster containing $0$ in a Bernoulli bond percolation model with parameter $p$, is exponentially decreasing \cite{HDT}. So for $\alpha>0$ small enough, $N^{0}$ admits an $\alpha$-exponential moment; i.e. $E(e^{\alpha N^{0}})<+\infty$. We deduce that

\begin{eqnarray*}
Z_{n}^{\text{wired}}(\log(1-e^\mu)+\alpha,\mu) &\le & E(e^{\alpha N^{0}})^{\#\partial \L_n},
\end{eqnarray*}

which implies that

\begin{eqnarray*}
\PP(\log(1-e^\mu)+\alpha,\mu)&=&\lim_{n\to \infty} \frac{ \log(Z_{n}^{\text{wired}}(\log(1-e^\mu)+\alpha,\mu))}{\# \EE_n}\\
&\le & \lim_{n\to \infty} \frac{\#\partial \L_n}{\# \EE_n}  \log\left(E(e^{\alpha N^{0}})\right) \\
&=&0.
\end{eqnarray*}

\end{proof}

\subsection{Lower and upper bounds for the threshold}\label{Section_bounds}

In this section we provide a lower and upper bounds for $\lambda^*(\mu)$ by using the convexity of the function $\PP$. 

Let us start with the upper bound.

\begin{prop}\label{superupperbound}
For $\mu\le \log(1-p_c(d))$ then $\lambda^*(\mu)=\log(1-e^\mu)$ and for  $\mu > \log(1-p_c(d))$ then  
$$\lambda^*(\mu)\le\log(p_c(d))-\frac{1-p_c(d)}{p_c(d)}\Big(\mu- \log(1-p_c(d))\Big).$$
\end{prop}
\begin{proof}

Let us recall the definition of $\lambda^*(\mu)$ in \eqref{definitionpcd} which ensures that the curve $\mu\to (\mu,\lambda^*(\mu))$ in $\R^2$ is in fact the boundary of the convex set $\{(\lambda,\mu)\in\R^2, \PP(\lambda,\mu)=0\}$. By theorem \ref{theoremexplicit}, this curve is explicit for values $\{\mu< \log(1-p_c(d))\}$ and equals to the curve $\mu\to (\mu,\log(1-e^\mu))$. That proves the first part of the proposition.

By convexity of the set $\{(\lambda,\mu)\in\R^2, \PP(\lambda,\mu)=0\}$ all the curve is included in the half plan in $\R^2$ delimited by the tangent line at the point $(\log(p_c(d),\log(1-p_c(d))$. The equation of this line is $\mu\mapsto \log(p_c(d))-\frac{1-p_c(d)}{p_c(d)}(\mu- \log(1-p_c(d)))$ and the proposition follows (see Figure \ref{figdiagram}).

\end{proof}

Let us now give the lower bound

\begin{prop} \label{superlowerbound}
For $\mu\le \log(1-1/(2d-1))$ then $\lambda^*(\mu)\ge \log(1-e^\mu)$ and for  $\mu\ge \log(1-1/(2d-1))$ then  

$$\lambda^*(\mu)\ge -\log(2d-1)+(2d-2)\Big(\log(1-1/(2d-1))-\mu\Big).$$
\end{prop}

\begin{proof}

As in the proof of Proposition \ref{superupperbound} we use the convexity of the function $\PP$. Before we need to control its increases.

\begin{lemm}
There exists a constanct $C>0$ such that for each $n\ge 1$ 
$$ \frac{\partial   \log(Z_{n}^{\text{wired}}(\lambda,\mu))}{\partial \mu} \le  (2d-2) \frac{\partial  \log( Z_{n}^{\text{wired}}(\lambda,\mu))}{\partial \lambda}+ Cn^{d-1}.$$

\end{lemm}

\begin{proof}

Let us recall formulas \eqref{increasespressure} on the derivatives of $\log( Z_{n}^{\text{wired}})$. A discrete isoperimetrical  inequality on $\Z^d$ ensures that any allowed configuration $\o$ in $\A$ satisfies $\partial N(\o)\le (2d-2)N+ 2d$. Taking accound the boundary effects, there exists a constant $C>0$ such that

$$\partial N_{\L_n}^{\text{wired}}\le (2d-2)N_{\L_n}+Cn^{d-1}.$$ 

The proof of the lemma follows.
\end{proof}

By finite variation increasing, we deduce that for any $(\lambda,\mu)\in\R$ and any $t\ge 0$

$$\frac{1}{\#\EE_n}\log\Bigg(Z_{n}^{\text{wired}}\bigg(\lambda-t\Big( 2d-2+C\frac{n^{d-1}}{\#\EE_n}\Big),\mu+t\bigg)\Bigg)\le  \frac{1}{\#\EE_n}\log(Z_{n}^{\text{wired}}(\lambda,\mu)) $$
and passing to the limit we obtain

$$ \PP \Big(\lambda-t(2d-2),\mu+t\Big)\le \PP \Big(\lambda,\mu\Big). $$

Since the function $\PP(.,.)$ is positive we deduce that $\PP(\lambda,\mu)=0$ implies $ \PP(\lambda-t(2d-2),\mu+t)=0$ for all $t\ge 0$.  By Theorem \ref{theoremexplicit}, for any $\mu\in\R$ we have $\PP(\log(1-e^\mu),\mu)=0$ and therefore $ \PP(\log(1-e^\mu)-t(2d-2),\mu+t)=0$ for all $t\ge 0$. Optimizing with respect to $\mu\in\R$ and $t\ge 0$ we obtain that for $\mu\le \log(1-1/(2d-1))$, $\lambda^*(\mu)\ge \log(1-e^\mu)$ and for $\mu\ge  \log(1-1/(2d-1))$

$$\lambda^*(\mu)\ge -\log(2d-1)+(2d-2)\Big(\log(1-1/(2d-1))-\mu\Big).$$

 The proposition is proved (see Figure \ref{figdiagram}).

\end{proof}

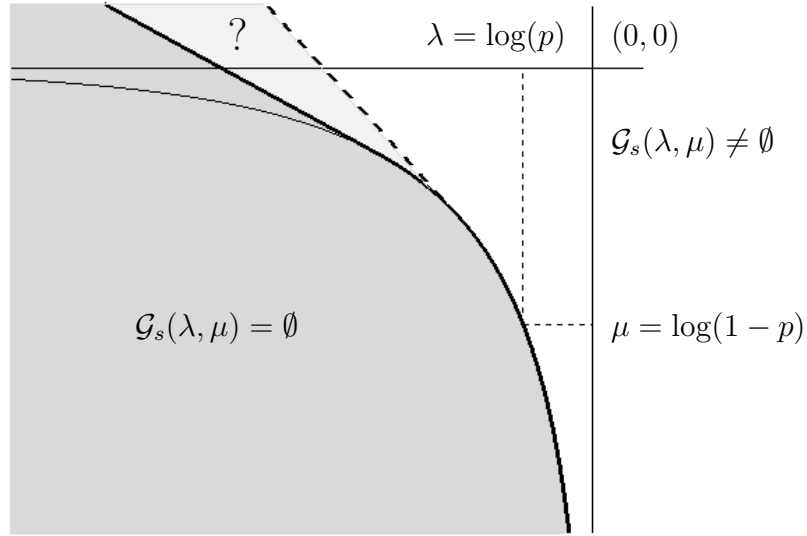
\begin{figure}
\begin{picture}(200,150)
\put(100,0){\includegraphics[scale=1.4]{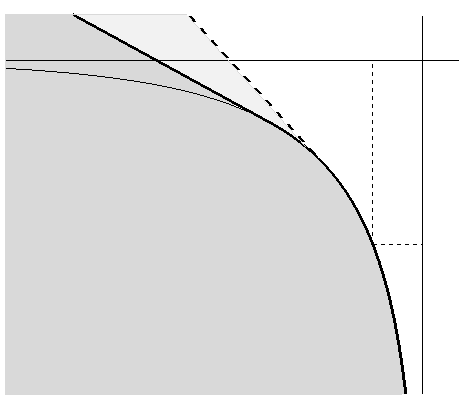}}
\put(150,80){$\G_s(\lambda,\mu)=\emptyset$}
\put(330,190){$(0,0)$}
\put(330,80){$\mu=\log(1-p)$}
\put(260,190){$\lambda=\log(p)$}
\put(185,188){\Large $?$}
\put(330,150){$\G_s(\lambda,\mu)\neq\emptyset$}
\end{picture}
\caption{\label{figdiagram} Phase diagram for $\G_s(\lambda,\mu)$ with respect to $(\lambda,\mu)$. The set  $\G_s(\lambda,\mu)$  is empty for $(\lambda,\mu)$ in the dark gray area and is not empty for $(\lambda,\mu)$ in the white area. The light gray area is the unknown part of the phase diagram. The critical curve $(\lambda^*(\mu),\mu), \mu\in\R)$ belongs to the light gray area and is on the boundary between the dark gray and white areas. The dotted line corresponds to the upper bound given in Proposition \ref{superupperbound} and the bold  line and curve correspond to the lower bound given in Proposition \ref{superlowerbound}. The thin black curve is parametrized by $(\log(p),\log(1-p)), p\in (0,1)$. All the curves are plotted in the case $d=2$.}
\end{figure}

\section{Proofs of main results}\label{Section_proofs}

In this section we give the proofs of results from Section \ref{Section_Results}. Several results are direct consequences of results obtained in the previous Section \ref{Section_2parModel} since $\LL(p),\LL_s(p),\G(p),\G_s(p)$ are nothing else than $\LL(\lambda,\mu),\LL_s(\lambda,\mu),\G(\lambda,\mu),\G_s(\lambda,\mu)$ for $\lambda=\log(p/(1-p))$ and $\mu=0$.

{\it  Theorem \ref{thegeointro} }. It is a direct corollary of Theorem \ref{TheoFCBMconnected}.

{\it Theorem \ref{theothtesholdintro}}. The threshold $p^*(d)$  is defined such that $\lambda^*(0)=\log(p^*(d)/(1-p^*(d)))$ leading to

\begin{equation}\label{transformation}
 p^*(d)=\frac{e^{\lambda^*(0)}}{1+e^{\lambda^*(0)}}
\end{equation} 
 
which belongs to $(0,1)$ since $\lambda^*(0)\in\R$ (Propositions \ref{superupperbound} and \ref{superlowerbound}). The second item in Theorem \ref{theothtesholdintro} is a direct consequence of the second item in Theorem \ref{TheoremThreshold}. For the first item it is a bit more delicate. Actually the first item in Theorem \ref{TheoremThreshold} ensures that for $p>p^*(d)$ there exists $P$ in $\in \G_s(p)$. At this stage it is not sure that $P$ is in $\LL_s(p)$. To this end we use a standard decomposition of $\G_s(p)$ in ergodic phases (see for instance \cite{georgii} for a general presentation). Indeed $\G_s(p)$ is a Choquet simplex where each extremal point is  the limit of finite volume Gibbs measures for deterministic boundary conditions. Precisely, since $\G_s(p)$ is not empty it contains at least one extremal point $P$ and therefore for $P$ almost every $\o$

$$\{P\}=\LL^{\text{bc}(\o)}(p)=\LL_s^{\text{bc}(\o)}(p).$$

 It is enough to ensure the existence of $P$ in $\in\LL_s(p)\cap \G_s(p)$ with $P(\A)=1$.

{\it Theorem \ref{theoboundsintro}}. The bounds \eqref{boundsintro} are direct consequences of identity \eqref{transformation} and Propositions \ref{superupperbound} and \ref{superlowerbound}.

\subsection{Proof of Theorem \ref{theouniciteintro}}
In this section we have to show that for $d=2$ and $p\ge 1/2$ the sets $\G(p)$ and $\G_s(p)$ are equal and reduced to a single element $\{P\}$ and moreover any $Q$ in $\LL(p)$ is a mixture between $P$ and  the vacuum state $\delta_{0^\EE}$. To this end we build an independent Bernoulli field which is dominated by any elements in $\G(p)$ or $\LL(p)$ (excepted the vacuum state). Based on disagreement arguments and a delicate coupling, we show that all $P,Q$ in $\G(p)$ or $\LL(p)$ (excepted the vacuum state) are identical provided that the dominated Bernoulli field percolates (or is at criticality). The starting point of our coupling construction is the following Lemma.

\begin{lemm} \label{probaminore}
Let  $p\ge 0$ and $P\in\G(p)\cup\LL(p)$. Let $E\subset \EE$ be a finite subset of edges and $\tilde \o\in\A$ an allowed configuration. Let $e$ be an edge in $\EE\backslash E$. We assume that there exists an open edge $f$ in $\tilde\o_E$ having a common vertex with $e$. Then

\begin{equation}\label{localcondi}
 P(e  \text{ is open } | \tilde\o_E) \ge p.
 \end{equation}

\end{lemm}

\begin{proof}

Denoting by $A$ the event of configurations $\o\in\O$ such that $\o_E=\tilde\o_E$, inequality \eqref{localcondi} is equivalent to claim that for  $P\in\G(p)\cup\LL(p)$ 

\begin{equation}\label{localcondi2}
P(A\cap\{e  \text{ is open }\})\ge p P(A). 
\end{equation}

Let $n$ large enough such all vertices of edges in $E$ or $e$ are included in $\L_n$. Then for any boundary condition $\bc(n)$

$$ \P_p^{\L_n}(A\cap\{e  \text{ is open }\}|\A^{\bc(n)})\ge p \P_p^{\L_n}(A|\A^{\bc(n)}).$$
Passing to the limit when $n\mapsto \infty$, inequality \eqref{localcondi2} holds for any $P$ in $\LL(p)$. Similarly for any allowed configuration $\hat \o\in\A$,

$$ \P_p^{\L_n}(A\cap\{e  \text{ is open }\}| \hat \o_{\EE_{\L_n}^c} )\ge p \P_p^{\L_n}(A| \hat \o_{\EE_{\L_n}^c}).$$

By DLR equations \eqref{DLRintro} and by integrating $\hat \o$ with respect to $P\in\G(p)$ we recover  inequality \eqref{localcondi2}.

\end{proof}

Let us give an interpretation of the lemma. For any process $P\in\G(p)\cup\LL(p)$, during an exploration procedure declaring the states of edges, if we want to explore an new edge closed to an open edge, then this new edge is open with probability at least $p$.

Let us now prove the theorem. Let $P$ and $Q$ be two elements in $\G(p)\cup\LL(p)$. In the following, for all bounded cubes $\L\subset \Delta \subset \Z^2$, we  build a coupling with three processes $X$, $Y$, $Z$ on $\O$ such that $X\sim P$, $Y\sim Q$ and $Z\sim \B(p)^{\otimes \EE}$ having the following coupling  property:

  [Coupling property:]  {\it if $Z$ has an open connected path in $\Delta$ surrounding $\L$ then $X$ and $Y$ are identical in $\L$ or $X$ (or $Y$) is null in $\L$. }

Before proving that this coupling exists, let us show how it allows to prove that for $p$ larger than the bond percolation threshold on $\Z^2$ (i.e. $p\ge \frac 12$) then $P(.|\{0^\EE\}^c)$ and $Q(.|\{0^\EE\}^c)$ have the same distribution. The proof of the theorem follows. Indeed, for a fixed bounded set $\L\subset \Z^2$, the probability, that $Z$ has an open connected path in $\Delta$ surrounding $\L$, is going to one when $\Delta$ is going to $\Z^2$ (note that this fact is true  also at the critical point $p=\frac 12$).  Then realizations of  $P$ and $Q$ are either equal in $\L$ or one of them (or both) charges the null configuration in $\L$. But given the event $\{0^\EE\}^c$ the probability under $P$ or $Q$ that a realization charges the null configuration in $\L$, is going to zero when $\L$ is going to $\Z^2$. So the probability that realizations under  $P(.|\{0^\EE\}^c)$ and  $Q(.|\{0^\EE\}^c)$ are locally equal is going to one  when $\L$ is going to $\Z^2$. It is enough to ensure that the probability measures $P(.|\{0^\EE\}^c)$ and  $Q(.|\{0^\EE\}^c)$ are the same. Recall that for $P$ in $\G(p)$, by definition $P(0^\EE)=0$, and so the set $\G(p)$ is reduced to a single element $\{P\}$. Now for any $Q$ in $\LL(p)$ we have that $Q(.|\{0^\EE\}^c)=P$ which is enough to ensure the expected mixture representation. The theorem is proved.

Let us now give the construction of the coupling with the good properties mentioned above. We fix a collection of independent random variates $(U_e)_{e\in\EE}$ with uniform distribution on $[0,1]$ and, as usual, the process $Z$ is simply defined by $(\1_{U_e<p})_{e\in\EE^\Delta}$. Now we build $X$ and $Y$ using the following revealment algorithm. First we realize arbitrary  $X$ and $Y$ outside the window $\Delta$ (i.e. for any edge having at least one vertex in $\Delta^c$) with distributions $P$ and $Q$. The boundary conditions in $\Delta$ are now observable.

By Proposition  \ref{Prop_notbounded} any connected component produced under $P$ or $Q$ is almost surely unbounded. So if the realization of $X$ (or $Y$) outside $\Delta$ is null we extend $X$ (or $Y$) inside $\Delta$ with the null configuration as well. If the realization of $X$ (or $Y$) outside $\Delta$ is not null but satisfies the "free" boundary condition (no open edge at the boundary) then the realization of $P$ (or $Q$) inside $\Delta$ is necessary the null configuration. So we set $X$ (or $Y$) equals to the null configuration in $\Delta$ as well. In any case the coupling property occurs. 

Let us now build the coupling in the case where the boundary conditions for $X$ and/or $Y$ are not "free". We sample $X$ (or $Y$) inside $\Delta$ with the following revealment algorithm. We reveal the open edges one by one exploring the neighbourhood of open edges revealed before by the algorithm. We start from the boundary of $\Delta$ where open edges are present. When we explore a new edge $e$ (hitting necessary an open edge explored before), by lemma \ref{probaminore}, the probability that this edge $e$ is open is larger than $p$. So we declare that this edge is open in $X$ if the uniform variate $U_e$ is smaller than the expected probability. We note directly that this edge $e$ is open in $Z$ if it is open in $X$. That is the way to couple $X$ with $Z$.  The order that the Algorithm uses to choose the edges during the exploration is crucial and explained now. First we choose an a priori order on edges in $\EE_\Delta$ to organize the exploration. At each step of the algorithm we choose a new edge in the list of edges hitting an open edge explored before. We choose it following the a priori order but we force some extra rules. The main aim is the following.

{\it Main property of the Algorithm:} If the algorithm reveals a cluster of $Z$ in $\Delta$, then it is explored only on its exterior boundary. The interior is not revealed.

Let us recall what is the exterior boundary of a finite cluster of edges $C$. The topological cluster (i.e. the union of edge-segments) splits the space $\mathbb R^2$ in different connected components where only one is unbounded. The exterior boundary of $C$ is the set of the edges at the boundary of this unbounded connected component.   

Let us describe now the extra rules:

\begin{itemize} 
\item {\it Rule 1 (Clock-wise direction rule):} The algorithm always chooses the edges in the clockwise direction with respect to  the edges discovered before. It means that a new explored edge does not have a non-explored edge at the anticlockwise position. This rule prevents that the new explored edge is inside a cluster of $Z$. If a cluster of $Z$ is discovered, necessary the edge belongs to its exterior boundary.

\item {\it Rule 2 (Boundary cluster of $Z$ priority rule):} If a cluster of $Z$ is discovered, then the algorithm explores in priority the full exterior boundary of this cluster. Precisely, after discovering  the first edge of the cluster, the algorithm explores the edges around the new vertex at the boundary (in the clockwise direction as mentioned in the Rule 1). When a new edge at the boundary is discovered, it starts again a new exploration around the new vertex and so on. Actually the Algorithm explores the exterior boundary of the cluster in the clockwise direction without revealing any edges inside the cluster. Let us note that the algorithm can visit twice an exterior boundary edge if the cluster is very thinned at this edge.

\item {\it Rule 3 (No exploration inside cluster of $Z$ rule):} The Algorithm never explores inside clusters of $Z$. Thanks to Rules 1 and 2, the exterior of the boundary of a cluster of $Z$ is explored first during the run of the algorithm. We add the constraint that the algorithm will never explore latter the rest of the cluster. 

\end{itemize}

The Algorithm stops when all edges allowed to be explored have been explored. It is not difficult to see that the process $X$ (or $Y$) is build completely excepted inside the clusters of $Z$ met during the exploration. We sample arbitrary inside these clusters following the DLR description. This procedure does not depend on the distribution $P$ or $Q$. So if $X$ and $Y$ discover the same cluster of $Z$, we use the same sampling inside the cluster. That is the way to couple $X$ and $Y$. It is now not difficult to see that the "coupling property" holds with this construction. Indeed, if $Z$ has an open connected path in $\Delta$ surrounding $\L$ then this path is included in an open cluster with exterior boundary surrounding $\L$ as well. So if the explorations of $X$ and $Y$ meet this cluster of $Z$ then $X$ and $Y$ are identical by construction. If the exploration of $X$ (or $Y$) does not meet this cluster then $X$ (or $Y$) is null on $\L$.

\section*{Acknowledgement}

The author would like to thank Vincent Beffara for the fruitful discussions on the topic.  This work was supported in part by the Labex CEMPI (ANR-11-LABX-0007-01), the ANR project PPPP (ANR-16-CE40-0016) and by the CNRS GdR 3477 GeoSto.

\bibliography{bibFCPM}

\bibliographystyle{plain}

\end{document}